\documentclass{article}
\usepackage{fullpage}
\usepackage[colorlinks,backref]{hyperref}
\usepackage{graphicx}
\usepackage{amsmath,amsthm,amssymb,enumerate}
\usepackage[normalem]{ulem} 
\usepackage{euscript,mathrsfs}
\usepackage[left=3cm,right=3cm,top=3.5cm,bottom=3.5cm]{geometry}
\usepackage{color}
\usepackage{bm}
\usepackage{bbm}
\catcode`\@=11 \@addtoreset{equation}{section}

\catcode`\@=12
\usepackage{color}
\allowdisplaybreaks

\newtheorem{Theorem}{Theorem}[section]

\theoremstyle{corollary}

\theoremstyle{definition}

\newtheorem{Remark}[Theorem]{Remark}

\newcommand{\dx}{\,{\rm d} {x}}

\newcommand{\R}{\mathbb{R}}

\newcommand{\p}{\mathbb{P}}
\newcommand{\T}{\mathbb{T}}

\theoremstyle{plain}
\newtheorem{thm}{Theorem}[section]
\theoremstyle{plain}
\newtheorem{lem}[thm]{Lemma}
\newtheorem{proposition}[thm]{Proposition}
\newtheorem{cor}[thm]{Corollary}

\theoremstyle{definition}
\newtheorem{definition}{Definition}[section]
\newtheorem{remark}{Remark}[section]

\newtheorem*{maintheorem*}{Main Theorem}
\newtheorem*{maincorollary*}{Main Corollary}

\allowdisplaybreaks

\begin{document}


\title{Stochastic Degenerate Fractional Conservation Laws}

\author{Abhishek Chaudhary}

\date{\today}

\maketitle

\medskip

\centerline{ Centre for Applicable Mathematics, Tata Institute of Fundamental Research}
\centerline{P.O. Box 6503, GKVK Post Office, Bangalore 560065, India}\centerline{abhi@tifrbng.res.in}


\medskip
\begin{abstract}
We consider the Cauchy problem for a degenerate fractional conservation laws driven by a noise. In particular, making use of an adapted kinetic formulation, a result of existence and uniqueness of solution is established. Moreover, a unified framework is also established to develop the continuous dependence theory. More precisely, we demonstrate $L^1$-continuous dependence estimates on the initial data, the order of fractional Laplacian, the diffusion matrix, the flux function, and the multiplicative noise function present in the equation.
\end{abstract}

\section{Introduction} 
Stochastic degenerate parabolic-hyperbolic equations are one of the most important branches of nonlinear stochastic PDEs. Equations of this type model the phenomenon of convection-diffusion of ideal fluid in porous media. Therefore these are in great demand in fluid mechanics. The study of this type of model equations with non-local operator, considered in \cite{biler}, is motivated by the anomalous diffusion encountered in many physical model. Nonlocal operator appears in mathematical phenomena for fluid flows and acoustic propagation in porous media, viscoelastic materials, and pricing derivative securities in financial markets \cite{Blackledge}.  The addition of a stochastic noise to this physical model is completly natural as it represents external perturbations or a lack of knowledge of certain physical parameters. In this paper, we are interested in the well-posedness theory for the degenerate fractional conservation laws driven by a Brownian noise in any space dimension. A formal description of our problem requires a filtered probability space $(\Omega,\mathcal{F},\mathbb{P},({\mathcal{F}}_t)_{t\,\ge\,0})$. We consider the following stochastic degenerate fractional conservation laws 
\begin{align}
	\label{1.1}
	\begin{cases} 
		d u(x,t)+\mbox{div}(F(u(x,t)))d t +g_x^\lambda[u(x,t)]dt=\mbox{div}\big(A(u)\nabla u\big)dt+\Phi(x, u(x,t))\,dB(t)& \text{in}\,\,\,{\Pi}_T, \vspace{0.1cm}\\
		u(0,x) = u_0(x), &\text{in}\,\,\,\,\, \mathbb{T}^N,
	\end{cases}
\end{align}
where ${\Pi}_T=\T^N\times (0,T)$ with $\T^N$ is the $N$-dimesnsional torus and $T\,\textgreater\, 0$ fixed, $u_0 : \Omega\times\T^3\to\mathbb{R}$ is the given initial random variable, $F:\R\to \R^N$ is a given (sufficiently smooth) vector valued flux function, $A$ is a diffusion matrix  (possibly degenerate) and
$g_x^\lambda$ denotes the fractional Laplace operator $(-\Delta)^\lambda$
of order $\lambda\in  (0,1)$. Note that $B$ is a cylindrical Wiener process, $B=\sum_{k\ge1}w_k \gamma_k$ , where the coefficients $w_k$ are independent Brownian processes and $(\gamma_k)_{k\ge1}$ is a complete orthonormal system in a Hilbert space $\mathfrak{X}$ and $ \Phi : L^2(\mathbb{T}^N )\to L_2(\mathfrak{X}; L^2(\mathbb{T}^N ))$ is $L_2(\mathfrak{X}; L^2(\mathbb{T}^N ))$-valued function, where $L_2(\mathfrak{X};L^2(\mathbb{T}^N))$ denotes the collection of Hilbert-Schmidt operators from $\mathfrak{X}$ to $L^2(\mathbb{T}^N )$(see Section \ref{Sec2} for the complete list of assumptions).
\subsection{Earlier works }
In the absence of non-local term along with the case $\Phi=0$ and $\,A=0$, the equation \eqref{1.1} becomes a standard
conservation laws. In deterministic set-up, entropy solution for conservation laws was studied by Kruzhkov \cite{Kruzhkov} and established well-posedness in $L^\infty$-framework. The entropy formulation of parabolic-hyperbolic problem involving Laray-Lions type operator has been treated by Carrillo \cite{Carrillo}. We refer Bendahmane $\&$ Karlsen \cite{Karlsen} for the more delicate anisotropic diffusion case. We also mention works of Alibaud\cite{Alibaud}, Cifani et.al. \cite{Cifani, jack} for deterministic fractal conservation laws and fractional degenerate convection-diffusion equations, respectively. The concept of kinetic solution was first introduced by Lions, Perthame, and Tadmor in \cite{lions}  for the scalar conservation laws. Chen $\&$ Perthame \cite{chen} developed a well-posedness theory for general degenerate parabolic-hyperbolic equations with non-isotropic nonlinearity. On the other hand, in stochastic set-up, the fundamental work of Kim \cite{Kim} who defined entropy solutions for the stochastic conservation-laws and established well-posedness theory to one-dimensional conservation laws that are driven by additive Brownian noise and Vallet $\&$ Wittbold \cite{Vallet} to multidimensional Dirichlet problem. However, when the noise is of mutiplicative type we refer to Feng $\&$ Nualart \cite{Feng} for one-dimensional balance laws. Debussche $\&$ Vovelle \cite{Debussche} introduced kinetic formulation of stochastic conservation laws and as result they were able to establish the well-posedness of multidimensional stochastic balance laws via kinetic approach. The more delicate anisotropic diffusion case has been treated by Debussche et.al. \cite{vovelle} using, in particular, the insight from the work \cite{hofmanova} of Hofmanova. We also refer the work of G. Lv et al. \cite{Lv} for stochastic nonlocal conservation laws in whole space. A number of authors have contributed in the area of stochastic conservation laws and we mention the works of Biswas et al. \cite{Biswas}, Koley et al. \cite{UKoley,Koley1}, Bhauryal et al. \cite{neeraj,Neeraj2}. In the stochastic setup, the main idea is to successfully capture the noise-noise interaction of the underlying problem since this plays an important role in the well-posedness theory for stochastic PDEs, for details see \cite{ neeraj, Neeraj2, Koley1,koley2013multilevel,Koley3, MKS01, K1, K2}. 
\subsection{Scope and outline of this paper}
Due to the presence of nonlinear flux term, degenerate diffusion term, and a nonlocal term in equation \eqref{1.1}, solutions to \eqref{1.1} are not necessarily smooth and weak solutions must be sought. In comparison to the notion of entropy solution introduced by Kruzkov\cite{Kruzhkov}, adapted notion of kinetic solutions seems to be better suited particularly for the degenerate fractional conservation laws, since this allows us to keep the precise structure of the parabolic dissipative measure, whereas in the case of entropy solution part of this information is lost and has to be recovered at some stage. To sum up, we aim at developing following results related to \eqref{1.1}:
 \begin{itemize}
 	\item[(1)] Drawing primary
 	motivation from \cite{sylvain, vovelle, hofmanova}, we establish the well-posedness theory of the kinetic solution to the Cauchy problem \eqref{1.1} by using vanishing viscosity method along with few a priori bounds. We also derive contraction principle, in which we employ a Kruzhkov doubling of variable technique and attempt to bound the difference of their kinetic solutions (see Theorem \ref{main theorem 1}).
 	\item[(2)]  For $\Phi=\Phi(u)$, making use of BV estimate, we also develop a unified framework to derive the continuous dependence estimates for the Cauchy problem \eqref{1.1} on initial data, order of fractional Laplacian, the flux function, and the multiplicative noise function (see Theorem \ref{main theorem 2}). Whenever $\Phi=\Phi(u,x)$ has dependency on the spatial variable $x$, BV estimates are no longer available even for the stochastic conservation laws (see \cite{chen}). 
 	\end{itemize}

   As an important part in this present study, formulation of kinetic solution is weak in space variable but strong in time variable, as proposed in \cite{sylvain}.  This formulation enjoys a c\'adl\'ag condition which ensures that solutions to \eqref{1.1} have continuous trajectories in $L^p(\mathbb{T}^N)$(see Corollaries \ref{continuity}\,\&\,\ref{continuity1}). This also
   allows us to obtain convergence of approximations for each time t (see remark \ref{fixed time}). In the case of hyperbolic scalar conservation laws, Dotti $\&$ Vovelle \cite{sylvain} defined a notion of generalized kinetic solution and obtained a comparison result showing that any generalized kinetic solution is actually a kinetic solution. To that context, difficulties are quite different in the our case. Indeed, due to structure of kinetic measure corresponding to non-local and second order terms, the comparison principle can be proved only for kinetic solutions and therefore, strong convergence of approximate solutions is needed in order to prove existence.  The main novelty of this work lies in successfully handling the nonlocal term (fractional Laplacian) in the proof of both comparison principle and continuous dependence estimate (see step 2 of both proofs).  We emphasize that the analysis presented in this manuscript differs significantly from the work in the stochastic nonlocal conservation laws on whole space \cite{Lv}, mainly because our formulation is different. Our method is inspired by the work \cite{sylvain}, where the authors used same notion of kinetic formulation for the stochastic conservation laws.

The paper is organized as follows. In Section~\ref{Sec2} and Section~\ref{3}, we give detials of basic setting, define the notion of kinetic solution, and state our main results, Theorem \ref{th2.10}\,$\&$ Theorem \ref{main theorem 2}. Section \ref{Sec3} is devoted to the proof of uniqueness part of Theorem \ref{th2.10} with the $L^1$-comparison principle, Theorem \ref{th3.5}. The existence part of Theorem \ref{th2.10} is done in Section \ref{Sec4}. The remainder of the paper deals with  BV estimate and continuous dependence estimate (proof of Theorem \ref{main theorem 2}), given in Section \ref{section5}. At the end, in Appendix \ref{A}, derivation of kinetic formulation is given.
\section{Assumptions and preliminaries}\label{Sec2}
\subsection{Hypotheses.} We now give the precise assumptions on each of the term appearing in the equation \eqref{1.1}. 

\noindent
\textbf{Flux function:} Let  $F=(F_1,F_2,....,F_N):\mathbb{R}\to\mathbb{R}^N$ be  $C^2(\mathbb{R};\mathbb{R}^N)$-smooth flux function with a polynomial growth of its derivative, in the following sense: there exists $q^*\,\ge\,1$ , $C\,\ge\,0 $, such that
\begin{align}\label{F.1}
	|F'(\xi)|\,\le\,C (1+|\xi|^{q^*-1}),
\end{align}
\begin{align}\label{F.2}
	\sup_{|\zeta|\,\le\,\delta}|F'(\xi)-F'(\xi+\zeta)|\,\le\,C (1+|\xi|^{q^*-1})\,\delta,
\end{align}
\textbf{Diffusion matrix:}
The diffusion matrix
$A=\big(A_{ij}\big)_{i,j=1}^N:\mathbb{R}\to\mathbb{R}^{N\times N}$
is positive semidefinite and symmetric. Its square-root matrix, which is also positive semidefinite and symmetric, is denoted by $\sigma$. We assume that $\sigma$ is bounded and locally $\gamma$-H\"older continuous for some $\gamma\,\textgreater\,\frac{1}{2}$, i.e.
\begin{align}\label{holder}
	|\sigma(\xi)-\sigma(\zeta)|\,\le\,C|\xi-\zeta|^{\gamma}\,\,\,\,\forall\,\,\xi,\zeta\in\mathbb{R},\,|\xi-\zeta|\,\textless\,1.
\end{align}
\textbf{Stochastic term:} Let $(\Omega,\mathcal{F},\mathbb{P},({\mathcal{F}}_t),(w_k(t))_{k\ge1})$ be a stochastic basis with a complete, right continuous filtration. Let $\mathcal{P}_T$ indicate the predictable $\sigma$-algebra on $\Omega\times[0,T]$ associated to $(\mathcal{F}_t)_{t\ge0}$. The initial data is random variable i.e. $\mathcal{F}_0$-measurable and we assume $u_0\in L^p(\Omega;L^p(\mathbb{T}^N))$ for all $p\in[1,+\infty)$. In this setting we can assume without loss of generality that the $\sigma$-algebra $\mathcal{F}$ is countably generated and $(\mathcal{F}_t)_{t\ge0}$ is the filtration generated by the Wiener process $B$ and $u_0$. We define the canonical space $\mathfrak{X}\subset\mathfrak{X}_0$ via
$$\mathfrak{X}_0=\bigg\{v=\sum_{k\ge1}\lambda_k \gamma_k;\, \sum_{k\ge1}\frac{\lambda_k^2}{k^2}\,\textless\,\infty\bigg\}$$
endowed with the norm 
$$\| v\|_{\mathfrak{X}_0}^2=\sum_{k\ge1}\frac{\lambda_k^2}{k^2},\,\,\, v=\sum_{k\ge1}\lambda_k \gamma_k .$$
Remark that the embedding $\mathfrak{X}\hookrightarrow \mathfrak{X}_0$ is Hilbert-Schmidt. Moreover, $\mathbb{P}$- a.s. trajectories of $B$ are in $C([0,T];\mathfrak{X}_0)$. For each $u\in L^2(\mathbb{T}^N)$ we consider a mapping $\Phi: \mathfrak{X}\to L^2(\mathbb{T}^N)$ defined by $\Phi(u)\gamma_k = \beta_k(\cdot, u(\cdot))$. Thus we define
$$\Phi(x,u)=\sum_{k\ge1} \beta_k(x,u)\gamma_k ,$$
the action of $\phi(x,u)$ on $\gamma\in \mathfrak{X}$ being given by ${\langle \phi(x,u),\gamma \rangle}_\mathfrak{X}$. We assume $\beta_k\in C(\mathbb{T}^N \times\mathbb{R})$, with the bounds
\begin{align}\label{2.2}
	\beta^2(x,u)=\sum_{k\ge1}|\beta_k(x,u)|^2 \le D_0 (1+|u|^2), 
\end{align}
\begin{align}\label{2.3} \sum_{k\ge1}|\beta_k(x,u)-\beta_k(y,v)|^2=D_1(|x-y|^2+|u-v|h(|u-v|)),
\end{align}
where $x, y\in\mathbb{T}^N$, $u, v\in\mathbb{R}$ and $h$ is a non-decreasing continuous function on $\mathbb{R}_+$  satisfying, h(0)=0 and $0\le h(z)\le1$ for all $z\in\mathbb{R}_+$. 
\noindent
Assumption \eqref{2.2} imply that $ \Phi : L^2(\mathbb{T}^N )\to L_2(\mathfrak{X}; L^2(\mathbb{T}^N ))$,
where $L_2(\mathfrak{X};L^2(\mathbb{T}^N))$ refer to the collection of Hilbert-Schmidt operators from $\mathfrak{X}$ to $L^2(\mathbb{T}^N )$. Thus, for any predictable process $u \in L^2(\Omega; L^2(0, T ; L^2(\mathbb{T}^N )))$, the stochastic integral $t\to\int_0^t \Phi(u) dB$ is a well defined process taking values in $L^2(\mathbb{T}^N)$ (see \cite{prato}  for detailed construction).

\noindent
\textbf{Fractional term:} $g^\lambda$ is a fractional laplace operator $(-\triangle)^{\lambda}$ for $\lambda\in(0,1)$ properly defined at least on $C^\infty(\mathbb{T}^N)$ by 
\begin{align}\label{2.4}
	g_x^{\lambda}[\phi](x):=-P.V.\int_{\mathbb{R}^N}(\phi(x+z)-\phi(x))\mu(z)\,dz\,\,\, \forall\,\,x\in\mathbb{T}^N
\end{align}
where $\mu(z)=\frac{1}{|z|^{N+2\lambda}}\,\, z\,\ne\,0$ and $\mu(0)=0$ (see \cite{Roncal, Barriosa} for details).

\noindent
\subsection{Assumptions for continuous dependence estimate}
We are also interested to develop a general framework for the continuous dependence estimate. Our aim is to establish continuous dependence on the fractional exponent $\lambda$ and on the non-linearities, that is, on the flux function and noise coefficients. To achieve that, we proceed as follows. Consider the pair of the nonlinear equations:
\begin{align}
	\label{E.1}
	\begin{cases} 
		d u(x,t)+\mbox{div}(F(u(x,t)))d t +g_x^\lambda[u(x,t)]dt=\mbox{div}\big(A(u)\nabla u\big)dt+\Phi(u(x,t))\,dB(t)& x \in \mathbb{T}^N, \,t \in(0,T) \vspace{0.1cm}\\
		u(0,x) = u_0(x), & x  \in \mathbb{T}^N
	\end{cases}
\end{align}
\begin{align}\label{E.2}
	\begin{cases} 
		d v(x,t)+\mbox{div}(G(v(x,t)))d t +g_x^\beta[v(x,t)]dt=\mbox{div}\big(B(v)\nabla v\big)dt+\Psi(v(x,t))\,dB(t)& x \in \mathbb{T}^N, \,t \in(0,T) \vspace{0.1cm}\\
		v(0,x) = v_0(x), & x  \in \mathbb{T}^N
	\end{cases}
\end{align}
For continuous dependence estimate, in addition, we are assuming following assumptions on terms in above equations. For all $\xi, \zeta\, \in \mathbb{R}$,
\begin{align}
	\sup_{|\zeta|\,\le\,\delta}|F'(\xi)-F'(\xi+\zeta)|\,&\le\,C(1+|\xi|^{p_*-1})\delta^{\lambda_{F_1}},\label{flux1}\\
	\sup_{|\zeta|\,\le\,\delta}|G'(\xi)-G'(\xi+\zeta)|\,&\le\,C(1+|\xi|^{p^*-1})\delta^{\lambda_{G_1}},\label{flux2}\\
	|\Phi(\xi)-\Phi(\zeta)|^2\,&\le|\xi-\zeta|^{\lambda_{F_2}+1},\label{noise1}\\
|\Psi(\xi)-\Psi(\zeta)|^2\,&\le|\xi-\zeta|^{\lambda_{G_2}+1},\label{noise2}
\end{align}
\begin{align}
	\|F'-G'\|_{L^\infty},\|\Phi-\Psi\|_{L^\infty}\,\textless\,\infty,\label{mixed}
\end{align}
\begin{align}
	u_0, \, v_0 \in L^p(\Omega\times\mathbb{T}^N)\cap L^1(\Omega;BV(\mathbb{T}^N))
\end{align}
where $\lambda_{F_1}, \lambda_{F_2}, \lambda_{G_1}$ and $\lambda_{G_2}$ are positive constant. We still assume that $\Phi$, $\Psi$ have at most linear growth \eqref{2.2}.
The diffusion matries
$$A=\big(A_{ij}\big)_{i,j=1}^N, B=(b_{ij})_{i,j=1}^N:\mathbb{R}\to\mathbb{R}^{N\times N}$$
are symmetric and positive semidefinite. Its square-root matries, which are also symmetric and positive semidefinite, is denoted by $\sigma$, $\tau$ respectively. We assume that $\sigma$, $\tau$ is bounded and locally $\gamma_a$, $\gamma_b$-H\"older continuous for some $\gamma_a, \gamma_b\,\textgreater\,\frac{1}{2}$ respectively, i.e.
\begin{align}
	|\sigma(\xi)-\sigma(\zeta)|\,&\le\,C|\xi-\zeta|^{\gamma_a}\,\,\,\,\forall\,\,\xi,\zeta\in\mathbb{R}\,|\xi-\zeta|\,\textless\,1,\\
	|\tau(\xi)-\tau(\zeta)|\,&\le\,C|\xi-\zeta|^{\gamma_b}\,\,\,\,\forall\,\,\xi,\zeta\in\mathbb{R}\,|\xi-\zeta|\,\textless\,1.\label{last}
\end{align}
\subsection{Preliminary results on Young measure}
Here we state some results regarding Young measure theory.  we refer to \cite[Section 2]{sylvain} for proof of results.
\begin{definition}\textbf{(Young measure)} Let $(\mathcal{O}, \mathbb{F}, \lambda_1)$ be a finite measure space. A mapping $\mathcal{V}$ from $\mathcal{O}$ to $\mathcal{P}(\mathbb{R})$, the set of probability measures on $\mathbb{R}$, is said to be a Young measure if, for all $ h \in C_b(\mathbb{R})$, the map $y\to\mathcal{V}_y(h)$ from $\mathcal{O}$ into $\mathbb{R}$ is $\mathbb{F}$-measurable. We say that a Young measure $\mathcal{V}$ vanishes at infinity if for all $q\ge1$,
	$$\int_{\mathcal{O}} \int_{\mathbb{R}}|\zeta|^q d\mathcal{V}_y(\zeta)d\lambda_1(y)\,\textless\, \infty.$$
\end{definition}
\begin{definition}\textbf{(Kinetic function)} Let $(\mathcal{O}, \mathbb{F}, \lambda_1)$ be finite measure space. A measurable function $f:\mathcal{O}\times\mathbb{R}\to[0,1]$ is said to be a kinetic function, if there exists a Young measure $\mathcal{V}$ on $\mathcal{O}$ vanishing at infinity such that, for $\lambda_1$-a.e. $y\in \mathcal{O}$, for all $\xi\in\mathbb{R}$,
	$$f(y,\zeta)=\mathcal{V}_y(\zeta,\infty).$$
\end{definition}
\noindent
\begin{definition}[\textbf{Equilibrium}] 
	We say that $f$ is an equilibrium, if there exists a measurable function $u:\mathcal{O}\to\mathbb{R}$ such that $f(y,\zeta)=\mathbbm{1}_{u(y)\textgreater\zeta}$ almost everywhere, or, equivalently, $\mathcal{V}_y=\delta_{\zeta=u(y)}$ for almost every $y\in \mathcal{O}$.
\end{definition}
\begin{thm}[\textbf{Compactness of Young measures}]\label{th2.6} Let $(\mathcal{O},\mathbb{F},\lambda_1)$ be a finite measure space such that sigma algebra $\mathbb{F}$ is countably generated. Let $(\mathcal{V}^n)$ be a sequence of Young measures on $\mathcal{O}$ satisfying uniformly for some $q\ge1$,
	\begin{align}\label{2.11}
		\sup_n\int_{\mathcal{O}}\int_\mathbb{R}|\zeta|^q\,d\mathcal{V}_z^n(\zeta)\,d\lambda_1(z)\,\textless\,+\infty.
	\end{align}
	Then there exists a Young measure $\mathcal{V}$ on $\mathcal{O}$ and a subsequence still denoted $(\mathcal{V}^n)$ such that, for all $h\in L^1(\mathcal{O})$, for all $g\in C_b(\mathbb{R})$,
	\begin{align}\label{2.12}
		\lim_{n\to +\infty}\int_{\mathbb{O}} h(y)\int_{\mathbb{R}}g(\zeta)d\mathcal{V}_z^n(\zeta)d\lambda_1(y)&=\int_{\mathcal{O}} h(y)\int_{\mathbb{R}}g(\zeta)d\mathcal{V}_y(\zeta)d\lambda_1(y).
	\end{align}
\end{thm}
\begin{cor}[\textbf{Compactness of kinetic functions}]\label{th2.7}Let $(\mathcal{O},\mathbb{F},\lambda_1)$ be a finite measure space such that $\mathbb{F}$ is countably generated. Let $(f_n)$ be a sequence of kinetic functions on $\mathcal{O}\times\mathbb{R}:f_n(y,\zeta)=\mathcal{V}_y^n(\zeta,+\infty)$ where $\mathcal{V}^n$ are Young measures on $\mathcal{O}$ satisfying for some $q\ge\,1$,
		\begin{align}\label{2.11}
		\sup_n\int_\mathcal{O}\int_\mathbb{R}|\zeta|^q\,d\mathcal{V}_y^n(\zeta)\,d\lambda_1(y)\,\textless\,+\infty.
	\end{align} Then there exists a kinetic function $f$ on $\mathcal{O}\times\mathbb{R}$ (related to the Young measure $\mathcal{V}$  by formula $f(y,\zeta)=\mathcal{V}_y(\zeta,+\infty)$) such that, up to a subsequence , $f_n\rightharpoonup f$ in $L^\infty(\mathcal{O}\times\mathbb{R})$ weak-*.
\end{cor}
\begin{lem}[\textbf{Convergence to an equilibrium}]\label{th2.8}Let $(\mathcal{O},\mathbb{F},\lambda_1)$ be a finite measure space, Let $q\,\textgreater\,1$. Let $(f_n)$ be a sequence of kinetic functions on $\mathcal{O}\times\mathbb{R}$: $f_n(y,\zeta)=\mathcal{V}_y^n(\zeta,+\infty)$ where $\mathcal{V}^n$ are Young measures on $\mathcal{O}$ satisfying for some $q\textgreater\,1$,
	 	\begin{align}\label{2.11}
		\sup_n\int_\mathcal{O}\int_\mathbb{R}|\zeta|^q\,d\mathcal{V}_y^n(\zeta)\,d\lambda_1(y)\,\textless\,+\infty.
	\end{align}
Let $f$ be a kinetic function on $\mathcal{O}\times\mathbb{R}$ such that $f_n\rightharpoonup f$ in $L^\infty (\mathcal{O}\times\mathbb{R})$ weak-*. Assuming that $f$ is an equilibrium, $f(y,\zeta)=\mathbbm{1}_{u(y)\textgreater\zeta}$, and letting
	$$u_n(y)=\int_{\mathbb{R}}\zeta d\mathcal{V}_y^n(\zeta)$$ then, for all $1\le p\,\textless\, q$, $u_n\to u $ in $L^p(\mathcal{O})$.
\end{lem}
\subsection{Sobolev space} Let us denote sobolev space $H^{\lambda}(\mathbb{T}^N)$, for $\lambda\in\mathbb{R}^+$, as the subspace of $L^2(\mathbb{T}^N)$ for which the norm
$$\|u\|_{H^{\lambda(\mathbb{T}^N)}}^2=\sum_{n\in\mathbb{Z}^N}\big(1+|n|^2\big)^{\lambda}|\hat{u}(n)|^2$$
is finite. Here $\hat{u}(n)$ denotes the Fourier coefficient.
\section{Definitions and main results} 
\label{3}
Here, we introduce the kinetic formulation to \eqref{1.1} as well as the basic definitions concerning the notion of kinetic solution.
\subsection{Random kinetic measure, kinetic solution}
\begin{definition}$\textbf{(Kinetic\,\, measure).}$\label{kinetic measure}
	A mapping ${m}$ from $\Omega$ to $\mathcal{M}^+(\mathbb{T}^N\times[0,T]\times\mathbb{R})$, the set of non negative measures over $\mathbb{T}^N\times[0,T]\times\mathbb{R}$, is said to be kinetic measure provided the following holds:
	\begin{enumerate}
		\item[(i)] $m$ is measurable in the following sense: for each $h\in C_0(\mathbb{T}^N\times[0,T]\times\mathbb{R})$ the mapping $m(h):\Omega\to\mathbb{R}$ is measurable,
		\item[(ii)] $m$ vanishes for large $\zeta$: if $\mathcal{B}_{R}^c=\{\zeta\in\mathbb{R}; |\zeta|\ge R\}$, then 
		$$\lim_{R\to\infty}\mathbb{E} m(\mathbb{T}^N\times[0,T]\times \mathcal{B}_{R}^c)=0.$$
	\end{enumerate}	
\end{definition}
\begin{definition}\label{kinetic solution}
	$\textbf{(Kinetic\,\, solution).}$ A $L^1(\mathbb{T}^N)$- valued stochastic process $(u(t))_{t\in[0,T]}$ is said to be a solution to \eqref{1.1} with initial datum $u_0$, if $(u(t))_{t\in[0,T]}$ and
	$f(t):=\mathbbm{1}_{u(t)\textgreater\zeta}$ have the following properties:
	\begin{enumerate}
		
		\item[1.] $u\in L_{\mathcal{P}_T}^p(\mathbb{T}^N\times[0,T]\times\Omega)\cap L^2(\Omega;L^2([0,T];H^{\lambda}(\T^N))),\,\,\,\,\forall\,\, p\in[1,+\infty)$ ,
		\item[2.] for all $\varphi\in C_c^2(\mathbb{T}^N\times\mathbb{R}),$ $\p$-almost surely, $t\to \langle f(t),\varphi\rangle$ is c\'adl\'ag,
		\item[3.]for all $p\in[1,+\infty),$ there exists $C_p\ge0$ such that 
		\begin{align}\label{2.5}
			\mathbb{E}(\sup_{0\le t\le T}\| u(t)\|_{L^p(\mathbb{T}^N)}^p)\le C_p,
		\end{align}
		\item [4.] $\mbox{div}\int_0^u\sigma(\zeta)d\zeta\,\in\,L^2(\Omega\times\,[0,T]\times\mathbb{T}^N),$
		\item [5.] for any $\varphi\in C_b(\mathbb{R})$ the following chain rule formula holds true: $\p$-almost surely,
		\begin{align}\label{chain}
			\mbox{div}\int_0^u \varphi(\zeta) \sigma(\zeta) d\zeta=\varphi(u)\mbox{div}\int_0^u\sigma(\zeta) d\zeta\,\qquad\,\text{in}\,\,\mathcal{D}'(\mathbb{T}^N)\,\,\text{a.e.}\, t\in [0,T],
		\end{align}
		\item[6.] Let $\eta_{1},\,\, \eta_2:\Omega\to\,\mathcal{M}^{+}\big(\mathbb{T}^N\times[0,T]\times\mathbb{R}\big)$ be defined as follows:  
		$$\eta_1(x,t,\xi)=\int_{\mathbb{R}^N}|u(x+z,t)-\zeta|\mathbbm{1}_{Conv\{u(x,t),u(x+z,t)\}}(\zeta)\mu(z)dz,$$
		and $$\eta_2(x,t,\zeta)=\big|\mbox{div}\int_0^u \sigma(\zeta)d\zeta\big|^2 \delta_{u(x,t)}(\zeta).$$
		There exists a random kinetic measure $m$ such that $\p$-a.s., $m\ge\eta_1+\eta_2$, and  for all $\varphi \in C_c^2(\mathbb{T}^N\times\mathbb{R})$, $t\in[0,T]$,
		\begin{align}\label{2.6}
			\langle f(t),\varphi \rangle &= \langle f_0, \varphi \rangle + \int_0^t\langle f(s),F'(\zeta)\cdot\nabla\varphi\rangle ds+\int_0^t\langle f(s),A(\zeta):D^2\varphi\rangle ds -\int_0^t\langle f(s), g_x^\lambda[\varphi]\rangle ds\notag\\
			&\qquad+\sum_{k=1}^\infty \int_0^t \int_{\mathbb{T}^N} \beta_k(x,u(x,s)\varphi(x,u(x,s) dx dw_k(s)\notag\\
			&\qquad+\frac{1}{2}\int_0^t\int_{\mathbb{T}^N}\partial_{\zeta}\varphi(x,u(x,s)) \beta^2 (x,u(x,s))dxds -m(\partial_{\zeta}\varphi)([0,t])
		\end{align}
		$\mathbb{P}$-almost surely,
		where
		$f_0(x,\zeta)=\mathbbm{1}_{u_0\textgreater\zeta}$,  $\beta^2 := \sum_{k\ge1} |\beta_k|^2$. 
	\end{enumerate}
\end{definition}
\noindent
Here we have used the brackets $\langle.,.\rangle$, to indicate the duality between $C_c^\infty(\mathbb{T}^N \times\mathbb{R})$ and the space of distributions over $\mathbb{T}^N\times\mathbb{R}$. We have used  the shorthand $m(\varphi)$ to indicate the Borel measure on $[0,T]$ defined by
$$m(\varphi):A\mapsto\int_{\mathbb{T}^N\times A\times \mathbb{R}}\varphi(x,\xi)dm(x,t,\zeta),\,\,\, \varphi \in C_b(\mathbb{T}^N\times\mathbb{R})$$
for all A Borel subset of $[0,T]$, and 
$$\text{Conv}\{a, b\} := (\text{min}\{a, b\},\text{max}\{a, b\}).$$
We have used the notation $A:B=\sum_{i,j}a_{ij}b_{ij}$ for two matries $A=(a_{ij})$, $B=(b_{ij})$ of the same size.
\subsection{The main results}In this subsection, we record the statements of main results. We have c\'adl\'ag condition  in formulation of kinetic solutions which ensure that almost surely, tracjectories of soluton $u$ are continuous in $L^p(\mathbb{T}^N)$. We have two results as follows.
\begin{thm}[\textbf{Existence and uniqueness}]\label{th2.10}\label{main theorem 1}
	Under the assumptions \eqref{F.1}-\eqref{2.3}, there exists a unique kinetic solution $(u(t))_{t\in[0,T]}$ to \eqref{1.1} which has $\p$-almost surely continuous trajectories in $L^p(\mathbb{T}^N)$, for all $p\in[1,+\infty)$. Moreover, if $(u_1(t))_{t\in[0,T]}, (u_2(t))_{t\in[0,T]}$ are kinetic solutions to \eqref{1.1} with initial data $u_{1,0}$ and $u_{2,0}$, respectively, then for all $t\in[0,T]$,
	\begin{align}\label{contraction}
		\mathbb{E}\left\Vert u_1(t)-u_2(t)\right\Vert_{L^1(\mathbb{T}^N)}\le\mathbb{E}\left\Vert u_{1,0}-u_{2,0}\right\Vert_{L^1(\mathbb{T}^N)}.
	\end{align}
\end{thm}
We also develop a general framework for the continuous dependence estimate of kinetic solutions. Note that the $L^1$- contraction  \eqref{contraction} gives the continuous dependence on the initial data. However, we intend to establish continuous dependence on the order of fractional Laplacian, the flux function, the diffusion matrix and the multiplicative noise present in equation \eqref{1.1}.
\begin{thm}[\textbf{Continuous dependence estimate}]\label{main theorem 2}
	Let assumptions \eqref{flux1}-\eqref{last} holds. Let $(u(t))_{t\in[0,T]}$ be a kinetic solution to \eqref{E.1} with initial data $u_0$, and let $(v(t))_{t\in[0,T]}$ be a kinetic solution to \eqref{E.2} with initial data $v_0$.   Then, the following continuous dependence estimate holds: for all $t\in[0,T]$,
	\begin{align*}
		\mathbb{E}\int_{\mathbb{T}^N}&|u(x,t)-v(x,t)|dx
		\le\,C_T\,\bigg(\mathbb{E}\bigg[\int_{\mathbb{T}^N}|v_0(x)-u_0(x)|dx\bigg]+\|F'-G'\|_{L^\infty(\mathbb{R})}\\&\qquad+\bigg(\|\Phi-\Psi\|_{L^\infty(\mathbb{R})}+\sqrt{\int_{|z|\,\le\,r_1}|z|^2 d|\mu_{\lambda}-\mu_{\beta}|(z)}+\|\sigma-\tau\|_{L^\infty(\mathbb{R})}\bigg)^{\min\big\{\frac{1}{2},\frac{\lambda_{G_1}}{2}, \lambda_{G_2}, \frac{\gamma_b}{2}\big\}}\\&\qquad+\int_{|z|\,\textgreater\,r_1}\mathbb{E}\big(\|u_0(\cdot+z)-u_0\|_{L^1(\mathbb{T}^N)}+\|v_0(\cdot+z)-v_0\|_{L^1(\mathbb{T}^N)}\big)d|\mu_\lambda-\mu_{\beta}|(z).
	\end{align*}
	where constant $C\,\textgreater\,0$ $(\text{depending on}\,\, T, u_0, v_0, G )$. Here $d\mu_{\lambda}(z):= \frac{dz}{|z|^{N + 2 \lambda}}$.
	\end{thm}
\begin{remark}
	Througout this paper, the letter C to denote various generic constant. There are situations where constant may change from line to line, but the notation is kept unchanged, so long as it does not impact central idea.
\end{remark}
\section{Proof of uniqueness and continuity part of Theorem \ref{main theorem 1}}\label{Sec3}
\subsection{Left limit representation} Since we need some technical results to prove the contraction principle, we first state those technical results and then move on to the proof of uniqueness. Here we closely follow the approach of \cite{sylvain} and obtain a canonical property of kinetic solution, which is useful to show that the kinetic solution, $u$ has almost surely continuous trajactories. In the following proposition, we show that the almost surely property to be c\'adl\'ag is independent from test function, and limit from the left at any point $t_{*}\in(0,T]$ is also represented by a kinetic function.
\begin{proposition}\label{th3.1} Let $u_0$ be a initial data. Let $(u(t))_{t\in[0,T]}$ be a solution to \eqref{1.1} with initial data $u_0$, then the following conditions hold,
	\begin{enumerate}
		\item[1.] there exists a measurable subset $\Omega_1\subset\Omega$ of full probability  such that, for all $\omega\in\Omega_1$, for all $\varphi\in C_c(\mathbb{T}^N\times\mathbb{R})$,$t\to\langle f(\omega,t),\varphi\rangle$ is c\'adl\'ag.
		\item[2.] there exists an $L^\infty(\mathbb{T}^N\times\mathbb{R};[0,1])$-valued process $(f^{-}(t))_{t\in(0,T]}$ such that: for all $t\in(0,T]$, for all $\omega\in\Omega_1$ for all $\varphi\in C_c^2(\mathbb{T}^N\times\mathbb{R})$, $f^{-}(t)$ is a kinetic function on $\mathbb{T}^N$ which represents the left limit of $s\to\langle f(s),\varphi \rangle$ at t:
		\begin{align}\label{3.1}
			\langle f^{-}(t),\varphi\rangle=\lim_{s\to t^{-}}\langle f(s),\varphi\rangle .
		\end{align}
	\end{enumerate}
\end{proposition}
\begin{proof} For a proof, one can follow similar lines as proposed in \cite[Proposition 2.10]{sylvain}.
	\end{proof}
\noindent
\textbf{Bounds:} By construction, we note that $\mathcal{V}^{-}=-\partial_{\zeta}f^{-}$ satisfies the following bounds: for all $\omega\in\Omega_1$,
	\begin{align}\label{3.3}\sup_{t\in[0,T]}\int_{\mathbb{T}^N}\int_{\mathbb{R}}|\zeta|^p d\mathcal{V}_{x,t}^{-}(\zeta)dx\le C_p(\omega),\,\,\,\mathbb{E}(\sup_{t\in[0,T]}\int_{\mathbb{T}^N}\int_{\mathbb{R}}|\zeta|^p d\mathcal{V}_{x,t}^{-}(\zeta)dx)\le C_p.
	\end{align}
	We obtain \eqref{3.3} using Fatou's lemma and $\eqref{2.5}$.
	
	\noindent
\textbf{Equation for $f^{-}$:} For all $\varphi\in C_c^2(\mathbb{T}^N\times\mathbb{R})$, $\mathbb{P}$-a.s., for all $t\in[0,T]$,
	\begin{align}\label{3.6}
		\langle f^{-}(t),\varphi \rangle &= \langle f(0),\varphi\rangle +\int_0^t\langle f(s),F'(\zeta)\cdot\nabla_x\varphi\rangle ds+\int_0^t\langle f(s),A:D^2\varphi\rangle ds-\int_0^t\langle f(s),g_x^\lambda[\varphi]\rangle ds\notag\\&\qquad+\sum_{k\ge 1}\int_0^t \int_{\mathbb{T}^N}\int_{\mathbb{R}}\beta_k(x,\xi)\varphi(x,\zeta)d\mathcal{V}_{x,s}(\zeta)dxdw_k(s)\notag\\
		&\qquad+\frac{1}{2}\int_0^t\int_{\mathbb{T}^N}\int_{\mathbb{R}}\beta^2(x,\zeta)\partial_{\zeta}\varphi(x,\zeta)d\mathcal{V}_{x,s}(\zeta)dx ds-m(\partial_{\zeta}\varphi)([0,t)).
	\end{align}
In particular, we have
\begin{align}\label{3.5}\langle f(t)-f^{-}(t), \varphi\rangle=-m(\partial_\zeta\varphi)(\{t\}).
\end{align}
It implies that outside the set of atomic points (at most countable points) of $A\mapsto m(\partial_{\zeta}\varphi)(A)$, we get $\langle f(t),\varphi\rangle=\langle f^{-}(t),\varphi\rangle.$ It shows that  $\p$-almost surely, $f=f^{-}$ a.e. $t\in[0,T]$.
\noindent

 In particilar, equation \eqref{3.6} gives us the following equation on $f^{-}$;  for all $\varphi\in C_c^2(\mathbb{T}^N\times\mathbb{R})$, $\mathbb{P}$-a.s., for all $t\in[0,T]$,
\begin{align}
	\langle f^{-}(t),\varphi \rangle &= \langle f(0),\varphi\rangle +\int_0^t\langle f^{-}(s),F'(\zeta)\cdot\nabla_x\varphi\rangle ds+\int_0^t\langle f^{-}(s),A:D^2\varphi\rangle ds-\int_0^t\langle f^-(s),g_x^\lambda[\varphi]\rangle ds\notag\\&\qquad+\sum_{k\ge 1}\int_0^t \int_{\mathbb{T}^N}\int_{\mathbb{R}}\beta_k(x,\zeta)\varphi(x,\zeta)d\mathcal{V}_{x,s}^{-}(\zeta)dxdw_k(s)\notag\\
	&\qquad+\frac{1}{2}\int_0^t\int_{\mathbb{T}^N}\int_{\mathbb{R}}\beta^2(x,\zeta)\partial_{\zeta}\varphi(x,\zeta)d\mathcal{V}_{x,s}^{-}(\zeta)dx ds-m(\partial_{\zeta}\varphi)([0,t)).
\end{align}

 We have c\'adl\'ag condition  in formulation of kinetic solutions which ensure that almost surely, tracjectories of soluton $u$ are right continuous in $L^p(\mathbb{T}^N)$. For that purpose, here we state the following result.
\begin{cor}\label{continuity} Let $(u(t))_{t\in[0,T]}$ be a solution to \eqref{1.1} with initial datum $u_0$. Then, for all $p\in[1,+\infty)$, for all $\omega\in\Omega_1$ (given in Proposition \ref{th3.1}), the map $t\mapsto u(t)$ from $[0, T]$ to $L^p(\mathbb{T}^N)$ is continuous from the right.
\end{cor}
\begin{proof}
	We refer to \cite[Corollary 2.13]{sylvain} for a proof.
	\end{proof}
\subsection{Doubling of variables}  As a next step towards the proof of uniqueness, we need a technical proposition relating two kinetic solutions to \eqref{1.1}. We will also use the following notation: If $f:X\times\mathbb{R}\to[0,1]$ is kinetic function, we denote by $\bar{f}$ the conjugate function $\bar{f}=1-f$. We denote by $f^+$ the right limit, which is simply $f$, that is $f^+(t):=f(t)$. From now on,  we will work with  two fixed representatives of $f$ $(f^{+}\,\text{and}\, f^{-})$  and we can take any of them in integral with respect to time or in a stochastic integral.  We need the following technical Proposition to prove uniqueness of kinetic solution. We follow similar lines as proposed in \cite{sylvain} for the proof of following proposition. Here, we give the details for the sake of completeness.
\begin{proposition}\label{th3.5}
	Let $(u_1(t))_{t\in[0,T]}$ and $(u_2(t))_{t\in[0,T]}$ be kinetic solutions to \eqref{1.1} with initial data $u_{1,0}$ and $u_{2,0}$, respectively and denote $f_1(t)=\mathbbm{1}_{u_1(t)\textgreater\,\xi}\,\,\&\,\, f_2(t)= \mathbbm{1}_{u_2(t)\,\textgreater\,\xi}$. Then, for all $t\in[0,T]$ and non-negative test functions $\theta\in{C}^\infty(\mathbb{T}^N)$,\, $\kappa\in{C_{c}^\infty(\mathbb{R})}$,  we have
	\begin{align}\label{app inequality}
		\begin{aligned}
		\mathbb{E}\bigg[\int_{(\mathbb{T}^N)^2}&\int_{\mathbb{R}^2}\theta(x-y)\kappa(\xi-\zeta)f_1^{\pm}(x,t,\xi)\bar {f}_2^{\pm}(y,t,\zeta)d\xi d\zeta dx dy\bigg]\\
		&\le\mathbb{E}\bigg[ \int_{(\mathbb{T}^N)}\int_{\mathbb{R}^2}\theta(x-y)\kappa(\xi-\zeta) f_{1,0} (x,\xi)\bar f_{2,0} (y,\zeta)d\xi d\zeta dx dy+ \mathcal{R}_{\theta}+ \mathcal{R}_{\kappa}+J+K\bigg], 
		\end{aligned}
	\end{align}
	where
	\begin{align*}
		\mathcal{R}_{\theta}&=\int_0^t\int_{(\mathbb{T}^N)^2}\int_{\mathbb{R}^2}f_1(x,s,\xi)\bar{f}_2(y,s,\zeta)(F'(\xi)-F'(\zeta))\kappa(\xi-\zeta)d\xi d\zeta\cdot\nabla\theta(x-y)dx dy ds,\\
		\mathcal{R}_{\kappa}\notag&=\frac{1}{2}\int_{(\mathbb{T}^N)^2}\theta(x-y)\int_0^t\int_{\mathbb{R}^2}\kappa(\xi-\zeta)\sum_{k\ge1}|\beta_k(x,\xi)-\beta_k(y,\zeta)|^2d\mathcal{V}_{x,s}^{1}\oplus\mathcal{V}_{y,s}^{2}(\xi,\zeta)dx dy ds,\\
		J&=-2\int_0^t\int_{(\mathbb{T}^N)^2}\int_{\mathbb{R}^2}f_1 (x,s,\xi)\bar {f}_2(y,s,\zeta)\kappa(\xi-\zeta)g_x^\lambda(\theta(x-y))d\xi d\zeta dx dy ds\notag\\
		&\qquad-\int_0^t\int_{(\mathbb{T}^N)^2}\int_{\mathbb{R}^2}f_1(x,s,\xi)\partial_{\xi}\kappa(\xi-\zeta)\theta(x-y)d\eta_{2,2}(y,s,\zeta)dxd\xi \notag\\
		&\qquad+\int_0^t\int_{(\mathbb{T}^N)^2}\int_{\mathbb{R}^2}\bar{f}_2(y,s,\zeta)\partial_{\zeta}\kappa(\xi-\zeta)\theta(x-y)d\eta_{1,2}(x,s,\xi)dy d\zeta,\\
		K&=\int_0^t\int_{(\mathbb{T}^N)^2}\int_{\mathbb{R}^2}f_1\bar{ f}_2(A(\xi)+A(\zeta)):D_x^2\theta(x-y)\kappa(\xi-\zeta)d\xi d\zeta dx dy ds\\
		&\qquad-\int_0^t\int_{\mathbb{T}^N}^2\int_{\mathbb{R}^2}\theta(x-y)\kappa(\xi-\zeta) d\mathcal{V}_{x,s}^1(\xi)dx d\eta_{2,3}(y,s,\zeta)\\
		&\qquad-\int_0^t\int_{\mathbb{T}^N}^2\int_{\mathbb{R}^2}\theta(x-y)\kappa(\xi-\zeta)d\mathcal{V}_{y,s}^2(\zeta)\,dy\,d\eta_{1,3}(x,s,\xi).
		\end{align*}
\end{proposition}
\begin{remark}
	Let us fix some notation corresponding kinetic solutions $u_1$ and $u_2$. Let $m_1$ and $m_2$ be are kinetic measures corresponding $u_1$ and $u_2$ respectively, satisfying $m_1 \ge \eta_{1,2}+\eta_{1,3}$ and $m_2\ge \eta_{2,2}+\eta_{2,3}$, $\mathbb{P}$-almost surely, where 
	\begin{align*}\eta_{1,2}(x,t,\xi)&=\int_{\mathbb{R}^N}|u_1(x+z,t)-\xi|\mathbbm{1}_{Conv\{u_1(x,t),u_1(x+z,t)\}}(\xi)\mu(z)dz,\\
	\eta_{1,3}(x, t, \xi)&=\big|\mbox{div}\int_0^{u_1(x,t)} \sigma(s)ds\big|^2 \delta_{u_1(x,t)}(\xi),\\
\eta_{2,2}(y,t,\zeta)&=\int_{\mathbb{R}^N}|u_2(y+z,t)-\zeta|\mathbbm{1}_{Conv\{u_2(y,t),u_2(y+z,t)\}}(\zeta)\mu(z)dz,\\
	\eta_{2,3}(y,t,\zeta)&=\big|\mbox{div}\int_0^{u_2(y,t)} \sigma(s)ds\big|^2 \delta_{u_2(y,t)}(\zeta).
	\end{align*}
	We can write $m_1=m_{1,1}+\eta_{1,2}+\eta_{1,3}$ and $m_2=m_{2,1}+\eta_{2,2}+\eta_{2,3}$ for some non negative measure $m_{1,1}$ and $m_{2,1}$ respectively.
\end{remark}
\begin{proof}
	We define $\bar{\beta}_1^2(x,\xi)= \sum_{k\ge1}|\beta_k(x,\xi)|^2$, and $\bar{\beta}_2^2(y,\zeta)=\sum_{k\ge1}|\beta_k(y,\zeta)|^2$. Let $\phi_1 \in C_c^\infty(\mathbb{T}_x^N\times\mathbb{R}_{\xi})$ and $\phi_2\in C_c^\infty(\mathbb{T}_y^N\times\mathbb{R}_{\zeta})$. For $f_1=f_1^+$ we have
	$$\langle f_1^+(t),\phi_1 \rangle = \langle \mu_1^*, \partial_\xi \phi_1 \rangle+M_1(t)$$
	with
	$$M_1(t)=\sum_{k\ge1} \int_0^t\int_{\mathbb{T}^N}\int_{\mathbb{R}} \beta_k(x,\xi)\phi_1(x,\xi) d\mathcal{V}_{x,s}^1(\xi) dx dw_k (s)$$
	and 
	\begin{align*}
		\langle \mu_1^*, \partial_{\xi}\phi _1 \rangle ([0,t])& =\langle f_{1,0},\phi_1 \rangle \delta_0([0,t])+\int_0^t \langle f_1, F'\cdot\nabla \phi_1 \rangle ds+\int_0^t\langle f_1(s),A(\xi):D_x^2\phi_1\rangle ds-\int_0^t \langle f_1 , g_x^\lambda[\phi_1]\rangle ds\\
		&\qquad+\frac{1}{2} \int_0^t \int_{\mathbb{T}^N}\int_{\mathbb{R}} \partial_{\xi} \phi_1 \bar{\beta}_1^2(x,\xi) d\mathcal{V}_{x,s}^1(\xi) dx ds -m_1(\partial_{\xi} \phi_1)([0,t]).
	\end{align*}
	We have $m_1(\partial_{\xi}\phi_1)(\{0\})=0$, and value of $\langle \mu_1^* ,\partial_{\xi} \phi_1\rangle (\{0\})$ is $\langle f_{1,0} , \phi_1 \rangle.$ Similarly,
	$$\langle \bar{f}_2^+(t),\phi_2 \rangle=\langle \bar{\mu}_2^*, \partial_{\zeta} \phi_2 \rangle([0,t])+\bar{M}_2(t)$$
	with 
	$$\bar{M}_2(t)=\sum_{k\ge1} \int_0^t\int_{\mathbb{T}^N}\int_{\mathbb{R}} \beta_k(y,\zeta)\phi_2(y,\zeta) d\mathcal{V}_{y,s}^1(\zeta) dx dw_k (s)$$
	and
	\begin{align*}
		\langle \bar{\mu}_2^*, \phi _2 \rangle ([0,t])& =\langle \bar{f}_{2,0}, \phi_1 \rangle \delta_0([0,t])+\int_0^t \langle\bar{ f}_2, F'\cdot\nabla \phi_2 \rangle ds+\int_0^t\langle \bar{f}_2(s),A(\zeta):D_y^2\phi_2\rangle ds-\int_0^t \langle \bar{f}_2 , g_y^\lambda[\phi_2]\rangle ds\\
		&\qquad-\frac{1}{2} \int_0^t \int_{\mathbb{T}^N}\int_{\mathbb{R}} \partial_{\zeta} \phi_2 \bar{\beta}_2^2(y,\zeta) d\mathcal{V}_{y,s}^2(\zeta) dx ds +m_2(\partial_{\zeta} \phi_2)([0,t]),	
	\end{align*}
	where $\langle \bar{\mu}_2^* ,\partial_{\zeta}\phi_2 \rangle (\{0\})=\langle \bar{f}_{2,0} , \phi_2 \rangle$. Let $\varphi(x,\xi, y, \zeta)= \phi_1(x,\xi)\phi_2(y,\zeta).$ Using It\^o formula for $M_1(t)\bar{M}_2(t)$, and integration by parts for functions of finite variation for $\langle \mu_1^* , \partial_{\xi} \phi_1 \rangle[0,t]\,\langle \bar{\mu}_2^*,\partial_{\zeta} \phi_2 \rangle ([0,t])$,(see \cite[Chapter 0]{Rev}) which gives
	\begin{align*}
		& \langle \mu_1^*, \partial_{\xi} \phi_1 ([0,t])\rangle\,\langle \bar{\mu}_2^* ,\partial_{\zeta}\phi_2 \rangle ([0,t])\\&=\langle \mu_1^* , \partial_{\xi} \phi_1 \rangle (\{0\})\,\langle \bar{\mu}_2^*, \partial_{\zeta} \phi_2 \rangle (\{0\})+\int_{(0,t]} \langle \mu_1^* ,\partial_{\xi} \phi_1 \rangle([0,s)) d \langle \bar{\mu}_2^*, \partial_{\zeta}\phi_2 \rangle (s)+\int_{(0,t]}\langle \bar{\mu}_2^* , \partial_{\zeta} \phi_2 ([0,s]) d\langle \mu_1^*, \partial_{\xi} \phi_1 \rangle (s)
	\end{align*}
	and the following formula
	\begin{align*}
		\langle \mu_1^* , \partial_{\xi} \phi_1 \rangle ([0,t]) \bar{M}_2(t)=\int_0^t \langle \mu_1^*, \partial_{\xi} \phi_1\rangle ([0,s]) d\bar{M}_2(s) + \int_0^t \bar{M}_2(s) \langle \mu_1^* , \partial_{\xi} \phi_1 \rangle (ds),
	\end{align*}
	which is easy to obtain since $\bar{F}_2$ is continuous and a similar formula for $\langle \bar{\mu}_2^* ,\partial_{\zeta}\phi_2 \rangle F_1(t)$, we get that
	$$\langle f_1^+(t), \phi_1 \rangle\,\langle \bar{f}_2^+(t),\phi_2 \rangle =\langle \langle f_1^{+}(t)\,\bar{f}_2^{+}(t), \varphi \rangle \rangle.$$ 
	It implies that
	\begin{align}\label{un}
		\mathbb{E}\langle \langle f_1^+(t)\,\bar{f}_2^+(t), \varphi \rangle\rangle=& \mathbb{E}\langle \langle f_{1,0} \bar{f}_{2,0}, \varphi\rangle\rangle\notag\\
		&+\mathbb{E}\int_0^t\int_{(\mathbb{T}^N)^2}\int_{\mathbb{R}^2} f_1 \bar{f}_2 (F'(\xi)\cdot\nabla_x + F'(\zeta)\cdot\nabla_y)\varphi d\xi d\zeta dx dy ds\notag\\
		&+\mathbb{E}\int_0^t\int_{\mathbb{T}^N}^2\int_{\mathbb{R}^2}f_1\bar{ f}_2(A(\xi)+A(\zeta)):D_x^2[\varphi]d\xi d\zeta dx dy ds\notag
		\\
		&-\mathbb{E}\int_0^t \int_{(\mathbb{T}^N)^2}\int_{\mathbb{R}^2} f_1 \bar{f}_2(g_x^\lambda+g_y^\lambda) [\varphi]\, d\xi d\zeta dx dy ds\notag\\
		&+\frac{1}{2}\mathbb{E}\int_0^t\int_{(\mathbb{T}^N)^2}\int_{\mathbb{R}^2}\partial_{\xi} \varphi \bar{f}_2(s) \bar{\beta}_1^2(x,\xi)\,d\mathcal{V}_{x,s}^1(\xi) d\zeta dx dy ds\notag\\
		&-\frac{1}{2} \mathbb{E} \int_0^t \int_{(\mathbb{T}^N)^2}\int_{\mathbb{R}^2}\partial_{\zeta} \varphi f_1(s) \bar{\beta}_2^2(y,\zeta)\, d\mathcal{V}_{y,s}^2(\zeta)\,d\xi dy dx ds\notag\\
		&-\mathbb{E}\int_0^t \int_{(\mathbb{T}^N)^2}\int_{\mathbb{R}^2} \bar{\beta}_{1,2}\varphi\,d\mathcal{V}_{x,s}^1(\xi)\,d\mathcal{V}_{y,s}^2(\zeta) dx dy ds\notag\\
		&-\mathbb{E}\int_{(0,t]}\int_{(\mathbb{T}^N)^2}\int_{\mathbb{R}^2} \bar{ f}_2^{+}(s)\partial_{\xi} \varphi dm_1(x,s,\xi) d\zeta dy\notag\\
		&+\mathbb{E}\int_{(0,t]}\int_{(\mathbb{T}^N)^2}\int_{\mathbb{R}^2}f_1^-(s) \partial_{\zeta} \varphi  dm_2(y,s,\zeta) d\xi dx	  
	\end{align}
	where $\bar{\beta}_{1,2}(x,y;\xi,\zeta):=\sum_{k\ge1 } \beta_k(x,\xi)\beta_k(y,\zeta)$ and $\langle\langle\cdot,\cdot \rangle\rangle$ indicates the duality distribution over $\mathbb{T}_x^N\times\mathbb{R}_{\xi}\times\mathbb{T}_y^N\times\mathbb{R}_{\zeta}$. Equation \eqref{un} also hold for any test function $\varphi\in C_c^\infty(\mathbb{T}_x^N\times\mathbb{R}_{\xi}\times\mathbb{T}_y^N\times\mathbb{R}_{\zeta})$ by a density argument. The assumption that $\varphi$ is compactly supported can be relaxed thanks to the condition at infinity on $m_i$ and $\mathcal{V}^i$, $i=1,2$. Using truncation argument for $\varphi$, we obtain that equation \eqref{un} is also  true if $\varphi \in C_b^\infty (\mathbb{T}_x^N\times\mathbb{R}_{\xi}\times\mathbb{T}_y^N\times\mathbb{R}_{\zeta})$ is compactly supported in a neighbourhood of the diagonal 
	$\big\{(x,\xi,x,\xi); x\in\mathbb{T}^N, \xi \in\mathbb{R}\big\}.$
	We then take $\varphi=\theta \kappa$ where $\theta=\theta(x-y), \kappa=\kappa(\xi-\zeta)$.
	We use the following identities
	$$(\nabla_x+\nabla_y)\varphi=0,\,\,\,\,\, (\partial_{\xi}+\partial_\zeta)\varphi=0,$$
	to obtain
	\begin{align*}
		&\mathbbm{E}\bigg[\int_{(\mathbbm{T}^N)^2}\int_{\mathbb{R}^2}\theta(x-y)\kappa(\xi-\zeta)f_1^{+}(x,s,\xi)\bar{f}_2^{+}(y,t,\zeta)d\xi d\zeta dx dy\bigg]\\
		&=\mathbb{E}\bigg[\int_{(\mathbb{T}^N)^2}\int_{\mathbb{R}^2}\theta(x-y)\kappa(\xi-\zeta) f_{1,0}(x,\xi)\bar{f}_{2,0}(y,\zeta)d\xi d\zeta dx dy+J+K'+\mathcal{R}_{\theta}+\mathcal{R}_{\kappa}+K\bigg],
	\end{align*} 	  
	where
	\begin{align*}
		K'&=\int_{(0,t]}\int_{\mathbb{T}^N}\int_{\mathbb{R}}f_1^{-}(x,s,\xi)\partial_\zeta\varphi\,\,\, dm_{2,1}(y,s,\zeta)-\int_{(0,t]}\int_{\mathbb{T}^N}\int_{\mathbb{R}}\bar{f}_2^{+}(y,s,\zeta)\,\,\partial_\xi \varphi\,\,\, dm_{1,1}(x,s,\xi)\\
		&=-\int_{(0,t]}\int_{\mathbb{T}^N}\int_{\mathbb{R}}f^{-}_1(x,s,\xi)\partial_{\xi} \varphi dm_{2,1}(y,s,\zeta)d\xi dx-\int_{(0,t]}\int_{\mathbb{T}^N}\int_{\mathbb{R}}f_2(y,s,\zeta)\partial_{\zeta} \varphi dm_{1,1}(x,s,\xi)dy d\zeta\\
		&=-\int_{(0,t]}\int_{\mathbb{T}^N}\int_{\mathbb{R}}\varphi d\mathcal{V}_{x,s}^{1,-}(\xi) dm_{2,1}(y,s,\zeta) dx-\int_{(0,t]}\int_{\mathbb{T}^N}\int_{\mathbb{R}}\varphi d\mathcal{V}_{y,s}^{2,+}(\zeta) dm_{1,1}(x,s,\xi)dy\\
		&\le\,0.
	\end{align*}
	Consequently we have the required estimate: for all $t\in[0,T]$
	\begin{align}
		&\mathbb{E}\bigg[\int_{(\mathbb{T}^N)^2}\int_{(\mathbb{R})^2}\theta(x-y)\kappa(\xi-\zeta)f_1^{\pm}(x,t,\xi)\bar {f}_2^{\pm}(y,t,\zeta)d\xi d\zeta dx dy\bigg]\notag\\
		&\qquad\leq\mathbb{E}\bigg[\int_{(\mathbb{T}^N)}\int_{(\mathbb{R}^2)}\theta(x-y)\kappa(\xi-\zeta)f_{1,0}(x,\xi)\bar f_{2,0} (y,\zeta)d\xi d\zeta dx dy + \mathcal{R}_{\theta}+ \mathcal{R}_{\kappa}+J+K\bigg].\notag
	\end{align}
\end{proof}
\begin{Remark}
	One can easily notice that, if $f_1^\pm=f_2^\pm$, then inequality \eqref{app inequality} holds pathwise, that is, $\p$-almost surely, for all $t\in[0,T]$,
	\begin{align}\label{app inequality1}
		\begin{aligned}
			\int_{(\mathbb{T}^N)^2}&\int_{\mathbb{R}^2}\theta(x-y)\kappa(\xi-\zeta)f_1^{\pm}(x,t,\xi)\bar {f}_2^{\pm}(y,t,\zeta)d\xi d\zeta dx dy\\
			&\le\int_{(\mathbb{T}^N)}\int_{\mathbb{R}^2}\theta(x-y)\kappa(\xi-\zeta) f_{1,0} (x,\xi)\bar f_{2,0} (y,\zeta)d\xi d\zeta dx dy+ \mathcal{R}_{\theta}+ \mathcal{R}_{\kappa}+J+K, 
		\end{aligned}
	\end{align}
where $\mathcal{R}_{\theta}, \mathcal{R}_{\kappa}, J,$ and $K$ are defined as above in Proposition \ref{th3.5}. For proof of this, there is no need to take expectation after use of It\^o formula and integration by part formula in the proof of Proposition \ref{th3.5}, since after the steps of approximation of test functions $\theta(x-y)$, and $\kappa(\xi-\zeta)$, the contributed martingale terms will cancel. 
\end{Remark}
\begin{thm}[\textbf{Contraction principle}]\label{th3.6}\label{comparison} If $(u_1(t))_{t\in[0,T]}$,\,$\&$\,$(u_2(t))_{t\in[0,T]}$\,are kinetic solutions to \eqref{1.1} with initial data $u_{1,0} $ and $u_{2,0}$ respectively, then for all $t\in[0,T]$,
	\begin{align}\label{contraction principle}
		\mathbb{E}\left\Vert u_1(t)-u_2(t)\right\Vert_{L^{1}(\mathbb{T}^N)}\le\mathbb{E}\left\Vert u_{1,0}-u_{2,0}\right\Vert_{\mathbb{L}^1(\mathbb{T}^N)}.
		\end{align}
	Moreover, let $(u(t))_{t\in[0,T]}$ be a solution to \eqref{1.1}, then there exists a $L^1(\mathbb{T}^N)$-valued process $(u^{-}(t))_{t\in[0,T]}$ such that $\mathbb{P}$-almost surely, for all $t\in[0,T]$, $f^{-}(t)=\mathbbm{1}_{u^{-}(t)\textgreater\xi}$. 
\end{thm} 
	\begin{proof} Let $(\theta_{\epsilon}),\,(\kappa_{\delta})$ be  approximations to the identity on $\mathbb{T}^N$ and $\mathbb{R}$, respectively, that is, $\theta\in C^\infty(\mathbb{T}^N)$ and $\kappa\in C_c^\infty (\mathbb{R})$ be symmetric non negative functions such that $\int_{\mathbb{T}^N} \theta(x) dx = 1 ,\,\, \int_{\mathbb{R}} \kappa(\xi) d\xi =1$ and supp$\kappa\subset(-1,1)$. We define
	$\theta_{\epsilon}=\frac{1}{\epsilon^N}\theta(\frac{x}{\epsilon}),$ and
	$\kappa_{\delta}(\xi)=\frac{1}{\delta}\kappa(\frac{\xi}{\delta})$. Then we follow proof of \cite[Theorem 3.2]{sylvain} to conclude 
	\begin{align*}
		\mathbb{E}\int_{\mathbb{T}^N}\int_{\mathbb{R}}f_1(x,t,\xi)\bar{f}_2(x,t,\xi)d\xi dx =\mathbb{E}\int_{(\mathbb{T}^N)^2}\int_{\mathbb{R}^2}\theta_{\epsilon}(x-y)\kappa_{\delta}(\xi-\zeta)f_1(x,t,\xi)\bar{f}_2(y,t,\zeta)d\xi d\zeta dx dy+\eta_{t}(\epsilon,\delta)\notag
	\end{align*}
	where $ \lim_{\epsilon,\delta\to 0}\eta_t(\epsilon,\delta)=0.$ 
	With regard to Proposition \ref{th3.5}, we need to find suitable bounds for terms I, J, K.
	We shall estimate each terms in the following several steps.
	
	\noindent
	\textbf{Step 1:}
	It follows from proof of \cite[Theorem 15]{Debussche} that $\p$-almost surely,
	\begin{align}\label{estimate1}
		|\mathcal{R}_\theta|\leq C'(\omega) \epsilon^{-1}\delta.
	\end{align}
where $C'(\omega)=\sup_{t\in[0,T]}\|u_1(w,t)\|_{L^{q^*}(\mathbb{T}^3)}^{q^*}+\|u_2(w,t)\|_{L^{q^*}(\mathbb{T}^3)}^{q^*}$. 

\noindent
	\textbf{Step 2:}
	In order to estimate the term J, we observe that
	\begin{align*}
		J&=-2\int_0^t\int_{(\mathbb{T}^N)^2}\int_{\mathbb{R}^2}f_1(x,s,\xi)\bar{ f}_2(y,s,\zeta)\kappa_{\delta}(\xi-\zeta)g_x^\lambda(\theta_{\epsilon}(x-y)))d\xi d\zeta dx dy ds\notag\\
		&\qquad+\int_0^t\int_{(\mathbb{T}^N)^2}\int_{\mathbb{R}^2}\partial_\zeta\kappa_{\delta}(\xi-\zeta)\theta_{\epsilon}(x-y)f_1(x,s,\xi)d\eta_{2,1}(y,s,\zeta)dx d\xi \notag\\
		&\qquad-\int_0^t\int_{(\mathbb{T}^N)^2}\int_{\mathbb{R}^2}\partial_\xi\kappa_{\delta}(\xi-\zeta)\theta_{\epsilon}(x-y)\bar{f}_2(y,s,\zeta)d\eta_{1,1}(x,s,\xi)dy d\zeta\notag\\
		&:=J_1+J_2,
	\end{align*}
	where

	\begin{align*}
			J_1&=-2\int_0^t\int_{(\mathbb{T}^N)^2}\int_{\mathbb{R}^2}f_1(x,s,\xi)\bar{ f}_2(y,s,\zeta)\kappa(\xi-\zeta)g_x^\lambda(\theta(x-y)))d\xi d\zeta dx dy ds,\\
		J_2&=\int_0^t\int_{(\mathbb{T}^N)^2}\int_{\mathbb{R}^2}\partial_\zeta\kappa_{\delta}(\xi-\zeta)\theta_{\epsilon}(x-y)f_1(x,s,\xi)d\eta_{2,1}(y,s,\zeta)dx d\xi \notag\\
		&\qquad-\int_0^t\int_{(\mathbb{T}^N)^2}\int_{\mathbb{R}^2}\partial_\xi\kappa_{\delta}(\xi-\zeta)\theta_{\epsilon}(x-y)\bar{f}_2(y,s,\zeta)d\eta_{1,1}(x,s,\xi)dy d\zeta.
	\end{align*}
	Now, we will write term $J_2$ in terms of $J_1$ and plus other term which has sign. Indeed,
	\begin{align*}
		&\int_0^t\int_{(\mathbb{T}^N)^2}\int_{\mathbb{R}^2}\partial_\xi\kappa_{\delta}(\xi-\zeta)\theta_{\epsilon}(x-y)\bar{f}_2(y,s,\zeta)d\eta_{1,1}(x,s,\xi)dy d\zeta\\
		&\quad=\int_0^t\int_{({\mathbb{T}^N})^2}\int_{\mathbb{R}^2}\int_{\mathbb{R}^N}|\tau_z u_1(x,s)-\xi|\mathbbm{1}_{Con\{u_1(x,s),\tau_z u_1(x,s)\}}(\xi) \partial_{\xi} \kappa_\delta(\xi-\zeta) \\
		&\qquad\qquad\qquad\qquad\qquad\qquad\qquad\qquad\qquad\qquad\qquad\qquad\quad\theta_{\epsilon}(x-y)\bar{f}_2(y,s,\zeta) \mu(z) dz d\xi d\zeta dx dy ds\\
		&\quad=\int_0^t\int_{(\mathbb{T}^N)^2}\int_{\mathbb{R}^{N+1}} \bigg\{\int_{u_1(x,s)}^{\tau_z u_1(x,s)}(\tau_z u_1(x,s)-\xi) \partial_{\xi} \kappa_\delta(\xi-\zeta) d\xi \\
		&\qquad\qquad\qquad\quad+\int_{\tau_z u_1(x,s)}^{u_1(x,s)}(\xi-\tau_z u_1(x,s))\partial_{\xi} \kappa_\delta(\xi-\zeta) d\xi \bigg\}\theta_\epsilon(x-y)\bar{f}_2(y,s,\zeta) \mu(z) dz d\zeta dx dy ds\\
		&\quad=\int_0^t\int_{\mathbb{R}^{N+1}}\int_{\mathbb{T}^N}\Bigg\{ \int_{\mathbb{T}^N\,\cap\,\{u_1(x,s)\,\le\,\tau_z u_1(x,s)\} }\bigg\{\big[(\tau_z u_1(x,s)-\xi)\kappa_\delta(\xi-\zeta)\big]_{u_1(x,s)}^{\tau_z u_1(x,s)}\\
		&\qquad\qquad+\int_{u_1(x,s)}^{\tau_z u_1(x,s)}\kappa_\delta(\xi-\zeta) d\xi \bigg\}+\int_{\mathbb{T}^N\cap\{\tau_zu_1(x,s)\le u_1(x,s)\}}\bigg\{\big[\xi-\tau_z u_1(x,s) \kappa_\delta(\xi-\zeta)\big]_{\tau_z u_1(x,s)}^{u_1(x,s)}\\
		&\qquad\qquad\qquad\qquad\qquad\qquad\qquad\quad-\int_{\tau_z u_1(x,s)}^{u_1(x,s)} \kappa_\delta (\xi-\zeta) d\xi \bigg\}\Bigg\} \theta_\epsilon(x-y)\bar{f}_2(y,s,\zeta) \mu(z) dz d\zeta dx dy ds\\
		&\quad=-\int_0^t\int_{({\mathbb{T}^N})^2}\int_{\mathbb{R}^{N+1}}\bigg\{(\tau_z u_1(x,s)- u_1(x,s))\kappa_{\delta}(u_1(x,s)-\zeta)\\
		&\qquad\qquad\qquad\qquad\qquad\qquad\qquad\quad+\int_{u_1(x,s)}^{\tau_z u_1(x,s)}\kappa_{\delta}(\xi-\zeta) d\xi \bigg\} \theta_\epsilon(x-y)\bar{f}_2(y,s,\zeta) \mu(z) dz d\zeta dx dy ds\\
		&\quad=-\int_0^t\int_{(\mathbb{T}^N)^2}\int_{\mathbb{R}^{N+1}}(\tau_z u_1(x,s)-u_1(x,s))\kappa_{\delta}(u_1(x,s)-\zeta)\theta_\epsilon(x-y)\bar{ f}_2(y,s,\zeta)\mu(z) dz d\zeta dx dy ds\\
		&\quad\quad+\int_0^t\int_{(\mathbb{T}^N)^2}\int_{\mathbb{R}^{N+1}}\int_{\mathbb{R}}\bigg\{\int_{-\infty}^{\tau_z u_1(x,s)}\kappa_{\delta}(\xi-\zeta)d\xi -\int_{-\infty}^{u_1(x,s)}\kappa_{\delta}(\xi-\zeta) d\xi \bigg\}\\
		&\qquad\qquad\qquad\qquad\qquad\qquad\qquad\qquad\qquad\qquad\qquad\qquad\qquad\quad\theta_\epsilon(x-y)\bar{ f}_2(y,s,\zeta) \mu(z) dz d\zeta dx dy ds\\
		&\quad=-\int_0^t\int_{(\mathbb{T}^N)^2}\int_{\mathbb{R}^N}\int_{\mathbb{R}}(\tau_z u_1(x,s)-u_1(x,s))\kappa_{\delta}(u_1(x,s)-\zeta)\theta_\epsilon(x-y)\bar{ f}_2(y,s,\zeta)\mu(z) dz d\zeta dx dy ds\\
		&\qquad+\int_0^t\int_{(\mathbb{T}^N)^2}\int_{\mathbb{R}^N}\int_{\mathbb{R}}\bigg(\int_{-\infty}^{ u_1(x,s)}\kappa_{\delta}(\xi-\zeta)d\xi\bigg) (\tau_z\theta_\epsilon(x-y)-\theta_\epsilon(x-y))\bar{ f}_2(y,s,\zeta) \mu(z) dz d\zeta dx dy ds\\
		&\quad=-\int_0^t\int_{(\mathbb{T}^N)^2}\int_{\mathbb{R}^N}\int_{\mathbb{R}}(\tau_z u_1(x,s)-u_1(x,s))\kappa_{\delta}(u_1(x,s)-\zeta)\theta_\epsilon(x-y)\bar{ f}_2(y,s,\zeta)\mu(z) dz d\zeta dx dy ds\\
		&\qquad-\int_0^t\int_{(\mathbb{T}^N)^2}\int_{\mathbb{R}^2} f_1(x,s,\xi)\kappa_{\delta}(\xi-\zeta) g_x^\lambda(\theta_\epsilon)(x-y) \bar{ f}_2(y,s,\zeta)\mu(z) dz d\xi d\zeta dx dy ds,
	\end{align*}
	where interchange of integration is justified by the fact that, $\mathbb{P}$-almost surely,
	\begin{align*}
		\int_0^t\int_{(\mathbb{T}^N)^2}\int_{\mathbb{R}^{N+2}} |\partial_{\xi}\kappa_{\delta}(\xi-\zeta)\theta_{\epsilon}(x-y)\bar{ f}_2(y,s,\zeta)\eta_{1,1}|& \mu(z)dz d\xi d\zeta dx dy ds\le\, C \|\eta_{1,1}\|_{L^1(\mathbb{T}^N\times\mathbb{R}\times[0,T])}.
	\end{align*}
	Similarly, for remaining term we have
	\begin{align*}
		&\int_0^t\int_{({\mathbb{T}^N})^2}\int_{\mathbb{R}^2} \partial_{\zeta} \kappa_{\delta}(\xi-\zeta) \theta_\epsilon(x-y) f_1(x,s,\xi) d\eta_{2,1}(y,s,\zeta) dx d\xi\\
		&=-\int_0^t\int_{(\mathbb{T}^N)^2}\int_{\mathbb{R}^N}\int_{\mathbb{R}}(\tau_z u_2(y,s)-u_2(y,s))\kappa_{\delta}(\xi-u_2(x,s))\theta_\epsilon(x-y) f_1(x,s,\xi)\mu(z) dz d\xi dx dy ds\\
		&\qquad+\int_0^t\int_{(\mathbb{T}^N)^2}\int_{\mathbb{R}^2} f_1(x,s,\xi)\kappa_{\delta}(\xi-\zeta) g_x^\lambda(\theta_\epsilon)(x-y) \bar{ f}_2(y,s,\zeta)\mu(z) dz d\xi d\zeta dx dy ds.
	\end{align*}
It conclude that
	\begin{align*}
		&J_2=-\int_0^t\int_{(\mathbb{T}^N)^2}\int_{\mathbb{R}}\int_{\mathbb{R}^N}\big(\tau_z u_2(y,s)-u_2(y,s))\big)(\kappa_{\delta} (\xi-u_2(y,s))\theta_{\epsilon}(x-y))f_1(x,s,\xi)\mu(z)dz dx dy d\xi ds\notag\\
		&\qquad+\int_0^t\int_{(\mathbb{T}^N)^2}\int_{\mathbb{R}^2}\bar{f}_2(y,s,\zeta) \kappa_{\delta}(\xi-\zeta)g_x^\lambda(\theta_{\epsilon})(x-y)f_1(x,s,\xi)\mu(z)dz dx dy d\zeta d\xi ds\\&\qquad+\int_0^t\int_{(\mathbb{T}^N)^2}\int_{\mathbb{R}}\int_{\mathbb{R}^N}\big(\tau_z u_1(x,t)-u_1(x,s))\big)\kappa_{\delta} (u_1(x,s)-\zeta)\theta_{\epsilon}(x-y)\bar{f}_2(y,t,\zeta)\mu(z)dz dx dy d\zeta ds\notag\\
		&\qquad+\int_0^t\int_{(\mathbb{T}^N)^2}\int_{\mathbb{R}}\int_{\mathbb{R}^N}f_1(x,s,\xi)\kappa_{\delta}(\xi-\zeta)g_x^\lambda(\theta_\epsilon)(x-y)\bar{f}_2(y,s,\zeta)\mu(z)dz dx dy d\zeta d\xi ds \\
		&:=-J_1+I_1
	\end{align*}
	where
	\begin{align*}
		I_1&=-\int_0^t\int_{(\mathbb{T}^N)^2}\int_{\mathbb{R}}\int_{\mathbb{R}^N}\big(\tau_z u_2(y,s)-u_2(y,s))\big)(\kappa_{\delta} (\xi-u_2(y,s))\theta_{\epsilon}(x-y))f_1(x,s,\xi)\mu(z)dz dx dy d\xi ds\notag\\
		&\qquad+\int_0^t\int_{(\mathbb{T}^N)^2}\int_{\mathbb{R}}\int_{\mathbb{R}^N}\big(\tau_z u_1(x,t)-u_1(x,s))\big)\kappa_{\delta} (u_1(x,s)-\zeta)\theta_{\epsilon}(x-y)\bar{f}_2(y,t,\zeta)\mu(z)dz dx dy d\zeta ds\notag\\
		&=\int_0^t\int_{(\mathbb{T}^N)^2}\int_{\mathbb{R}}\int_{\mathbb{R}^N}(\tau_z u_1(x,s)-u_1(x,s))\kappa_{\delta}(u_1(x,s)-\zeta)\theta_{\epsilon}(x-y)\bar{f}_2(y,s,\zeta)\mu(z)dz d\zeta dx dy ds\notag\\
		&\qquad-\int_0^t\int_{(\mathbb{T}^N)^2}\int_{\mathbb{R}}\int_{\mathbb{R}^N}(\tau_z u_2(y,s)-u_2(y,s))\kappa_{\delta}(\xi-u_2(y,s))\theta_{\epsilon}(x-y) f_1(x,s,\xi)\mu(z) dz d\xi dx dy ds\notag\\
		&=\int_0^t\int_{(\mathbb{T}^N)^2}\int_{\mathbb{R}^N}(\int_{u_2(y,s)}^{+\infty}\kappa_{\delta}(u_1(x,s)-\zeta)d\zeta)(\tau_z u_1(x,s)-u_1(x,s))\theta_{\epsilon}(x-y)\mu(z)dz dx dy ds\notag\\
		&\qquad-\int_0^t\int_{(\mathbb{T}^N)^2}\int_{\mathbb{R}^N}(\int_{-\infty}^{u_1(x,s)}\kappa_{\delta}(\xi-u_2(y,s))d\xi)(\tau_z u_2(y,s)-u_2(y,s))\theta_{\epsilon}(x-y)\mu(z) dz dx dy ds\notag\\
		&=\int_0^t\int_{(\mathbb{T}^N)^2}\int_{\mathbb{R}^N}\Phi_{\delta}(u_1(x,s)-u_2(y,s))[(\tau_z(u_1(x,s)\\
		&\qquad\qquad\qquad\qquad\qquad\qquad\qquad\qquad\qquad-u_2(y,s))-(u_1(x,s)-u_2(y,s))]\theta_{\epsilon}(x-y)\mu(z)dz dx dy ds\notag\\
		&=-\frac{1}{2}\int_0^t\int_{(\mathbb{T}^N)^2}\int_{\mathbb{R}^N}[\tau_z\Phi_{\delta}(u_1(x,s)-u_2(y,s))-\Phi_{\delta}(u_1(x,s)-u_2(y,s))]\\
		&\qquad\qquad\qquad\qquad\qquad\qquad\qquad[\tau_z(u_1(x,s)-u_2(y,s))-(u_1(x,s)-u_2(y,s))]\theta_{\epsilon}(x-y)\mu(z)dz dx dy \notag\\
		&\le0,\notag
	\end{align*}
	where $\Phi_{\delta}(w)=\int_{-\infty}^w\kappa_\delta(\xi)d\xi$ and note that $\Phi_{\delta}$ is non decreasing.
	Finally it shows clearly, $\p$-almost surely,
	\begin{align}\label{estimate2}
		J\,\le\,0.
	\end{align}
	\textbf{Step 3:}
	In order to estimate the term K, we  closely follow proof of \cite[Theorem 3.3]{vovelle}. We observe that
	\begin{align}
		K&=\mathbb{E}\int_0^t\int_{(\mathbb{T}^N)^2}\int_{\mathbb{R}^2} f_1 \bar{f}_2(\sigma(\xi)-\sigma(\zeta))^2:D^2_x\theta_{\epsilon}(x-y)\phi_{\delta}(\xi-\zeta)d\xi d\zeta dxdyds\notag\\
		&\qquad+2\mathbb{E}\int_0^t\int_{(\mathbb{T}^N)^2}\int_{\mathbb{R}^2}f_1\bar{f}_2\sigma(\xi)\sigma(\zeta):D^2_x\theta_{\epsilon}(x-y)\phi_{\delta}(\xi-\zeta)d\xi d\zeta dx dy ds\notag\\
		&\qquad-\mathbb{E}\int_0^t\int_{(\mathbb{T}^N)^2}\int_{\mathbb{R}^2}\theta_{\epsilon}(x-y)\phi_{\delta}(\xi-\zeta)d\mathcal{V}_{x,s}^1dxd\eta_{2,3}(y,x,\zeta)\notag\\
		&\qquad-\mathbb{E}\int_0^t\int_{(\mathbb{T}^N)^2}\int_{\mathbb{R}^2}\theta_{\epsilon}(x-y)\phi(\xi-\zeta)d\mathcal{V}_{y,s}^2(\zeta)dy d\eta_{1,3}(x,s,\xi)\notag\\
		&=K_1+K_2+K_3+K_4\notag
	\end{align}
	
	Since $\sigma$ is locally $\gamma$-H\"{o}lder continuous due to \eqref{holder}, it holds
	$$|K_1|\le C t\delta^{2\gamma}\epsilon^{-2}$$
	From the definition of the parabolic dissipative measure in Definition \ref{kinetic solution}, we have $K_2+K_3+K_4\le0$ (for detials see \cite[Theorem 3.3]{vovelle}). It conclude that $\p$-almost surely,
	\begin{align}\label{estimate3}
			|K|\le C(\omega) t\delta^{2\gamma}\epsilon^{-2}
	\end{align}

\noindent
	\textbf{Step 4:} In this step we estimate It\^o correction terms as follows:
	\begin{align}\label{estimate4}
		\mathcal{R}_{\kappa}&\le D_1\,\int_0^t\int_{(\mathbb{T}^N)^2}\theta_{\epsilon}(x-y)|x-y|^2\int_{\mathbb{R}^2}\kappa_{\delta}(\xi-\zeta)d\mathcal{V}_{x,s}^1d\mathcal{V}_{y,s}^2(\zeta)dx dy ds\notag\\
		&\qquad+D_1\,\int_0^t\int_{(\mathbb{T}^N)^2}\theta_{\epsilon}(x-y)\notag\times\int_{\mathbb{R}^2}\kappa_{\delta}(\xi-\zeta)|\xi-\zeta|h(|\xi-\zeta|)d\mathcal{V}_{x,s}^1(\xi) d\mathcal{V}_{y,s}^2(\zeta)dx dy ds\notag\\
		&\le D_1t\delta^{-1}\epsilon^2+D_1 t h(\delta).
	\end{align}
	\textbf{Step 5:}
	As a consequence of previous steps, we deduce for all $t\in[0,T]$,
	\begin{align*}
		\begin{aligned}
		\mathbb{E}\int_{\mathbb{T}^N}\int_{\mathbb{R}}f_{1}(x,t,\xi)f_2(x,t,\xi)d\xi dx
		&\le\mathbb{E}\int_{(\mathbb{T}^N)^2}\int_{\mathbb{R}^2}\theta_{\epsilon}(x-y)\phi_{\delta}(\xi-\zeta)f_{1,0}\bar{f}_{2,0}(y,\zeta)d\xi d\zeta dx dy \\&\qquad+C_T\big(\delta\epsilon^{-1}+\delta^{-1}\epsilon^2+h(\delta)+\delta^{2\gamma}\epsilon^{-2}\big)+\eta_t(\epsilon,\delta).\notag
		\end{aligned}
	\end{align*}
	Taking $\delta=\epsilon^a$ with $a\in(\frac{1}{\gamma},2)$ and letting $\epsilon\,\to\,0$ yields, for all $t\in[0,T]$
	\begin{align}\label{contraction principle 1}
		\mathbb{E}\int_{\mathbb{T}^N}\int_{\mathbb{R}}f_1(x, t, \xi)\bar{f}_2(x,t,\xi)d\xi dx\le\mathbb{E}\int_{\mathbb{T}^N}\int_{\mathbb{R}}f_{1,0}(x,
		\xi)\bar{f}_{2,0}(x,\xi)d\xi dx.
		\end{align}
	Since $$\int_{\mathbb{R}}\mathbbm{1}_{u_1(t)\textgreater\xi}\bar{\mathbbm{1}}_{u_2(t)\textgreater\xi}d\xi=(u_1(t)-u_2(t))^{+}.$$
	It implies that contraction principle \eqref{contraction principle} holds.
	
	\noindent For second remaining part, making use of inequality \eqref{app inequality1} and pathwise estimates \eqref{estimate1}-\eqref{estimate4}, we can similarly conclude  that $\p$-almost surely, for all $t\in[0,T]$,
	 \begin{align}\label{contraction principle 2}
		\int_{\mathbb{T}^N}\int_{\mathbb{R}}f^{-}(x,t,\xi)\bar{f}^{-}(x,t,\xi)d\xi dx\,\le\,\int_{\mathbb{T}^N}\int_{\mathbb{R}}f_{0}(x,\xi)\bar{f}_{0}(x,\xi)d\xi dx.
	\end{align}
	We have the identity $f_0(x,\xi)\bar{f}_0(x,\xi)=0$ and therefore, $\p$-almost surely, for all $t\in[0,T]$, $f^{-}(x,t,\xi)(1-f^{-}(x,t,\xi))=0$ a.e. $(x,\xi)$. The fact that $f^{-}$ is a kinetic function and then Fubini's theorem imply that, $\p$-almost surely, for all $t\in[0,T]$, there exists a set $S\subset\mathbb{T}^N$ of full measure such that, for $x\in S,f^{-}(x,t,\xi)\in\{0,1\}$ for a.e. $\xi\in\mathbb{R}$. Therefore, for all $t\in [0,T]$, there exists $u^{-}(t):\Omega\to L^1(\mathbb{T}^N)$ such that $\p$-almost surely, for all $t\in[0,T]$ $f^{-}(t)=\mathbbm{1}_{u^{-}(t)\textgreater\xi}$. 
	\end{proof}
\noindent
\textbf{Continuity in time:}\label{section 4.1}
As a consequence, we obtain the continuity of trajectories in $L^p(\mathbb{T}^N)$ whose proof is similar to the proof of  \cite[Corollary 3.3]{sylvain}.
\begin{cor}\label{continuity1}
	Let $u_0\in L^p(\Omega;L^p(\mathbb{T}^N))$ for all $p\in [1,+\infty)$. Then, the solution $(u(t))_{t\in[0,T]}$ to \eqref{1.1} with initial datum $u_0$ has $\p$-almost surely continuous trajectories in $L^p(\mathbb{T}^N)$.
\end{cor}
\begin{remark}[No atomic point]\label{remark 4.3} Since for all $t\in[0,T]$, $\p$-almost surely, $f(t)=f^{-}(t)$ (cf.\cite[Proposition 2.11]{sylvain}), then identity \eqref{3.5} implies that set of atomic point, ${B}_{at}=\{t\in[0,T];{\mathbb{P}}(\pi_\# {m}(\{t\})\textgreater 0)\textgreater 0\}$, is empty.
\end{remark}
\section{Proof of existence part of Theorem \ref{main theorem 1}}\label{Sec4}
In this section, first we prove the existence part of Theorem \ref{th2.10} for the initial condition $u_0\in L^p(\Omega;C^\infty (\mathbb{T}^N)).$
Here we apply the vanishing viscosity method, while using also some appropriately choosen approximations $F^\tau$ of $F$. These equations have weak solutions and consequently passage to the limit gives the existence of a kinetic solution to the original equation \eqref{1.1}. Neverthless, the limit argument is quite techincal and has to be done in several steps. 

Consider a truncation $(\mathcal{S}_\tau)$ on $\mathbb{R}$ and approximations $(\kappa_\tau)$  to the identity on $\mathbb{R}$. 
The regularization of $F$ is then defined in the following way
$$F_i^\tau(\zeta)= \big((F_i * \kappa_\tau)\mathcal{S}_\tau\big)(\zeta)\,\,\,\, i=1,...,N,$$
Consequently, we set $F^\tau=(F_1^\tau,...,F_N^\tau)$. It is clear that approximations $F^\tau$ is of class $C^\infty$ with the compact support therefore Lipschitz continuous. Also the polynomial growth of $F$ remains valid for $F^\tau$ and holds uniformly in $\tau$. 
\begin{align}\label{4.1}
	du^\tau(x,t)+\mbox{div}(F^\tau(u^\tau(x,t))&)d t +g_x^\lambda[u^\tau(x,t)]dt \notag \\
	&=\mbox{div}\big(A(u^\tau)\nabla u^\tau\big)+\tau\Delta u^\tau+ \Phi(x,u^\tau(x,t))dW(t)\,\, x \in \mathbb{T}^N,\,\, t \in(0,T).\\
	u^\tau(0)&=u_0.\notag
\end{align} 
There exists a unique weak solution $u^\tau$ to \eqref{4.1} such that
$$u^\tau\in L^2(\Omega;C([0,T];L^2(\mathbb{T}^N)))\cap L^2(\Omega;L^2(0,T;H^1(\mathbb{T}^N))).$$
\textbf{Existence of viscous solution:} One can follow simillar lines of  \cite[Section 7]{neeraj} for the proof of the existence of unique weak solution to \eqref{4.1}. It is based on a semi-implicit time-discretization. To complete all the steps in proof of existence of viscous solution as proposed in \cite[Section 7]{neeraj}, we have to use repeatedly following property regarding diffusion term, for all $u, v\in H^1(\T^3)$
$$\int_{\T^3}\big(\mbox{div}(A(u)\nabla u)-\mbox{div}(A(v)\nabla v)\big)(u-v)\dx=-\int_{\T^3}\big|\sigma(u)\nabla u-\sigma(v)\nabla v\big|^2\dx.$$
Since contributions from diffusion term give sign, making use of above identity we can easily estimate all the necessary uniform bounds to get compactness results.  

\noindent
\textbf{Convergnece of viscous solutions:} From the proof of the continuous dependence estimate as proposed in \eqref{final}, \eqref{H1} and \eqref{H2} (for detials see Section 6 in \cite{vovelle}), we conclude that for all $t\in[0,T]$,
\begin{align*}
	\mathbb{E}\bigg[\int_{\mathbb{T}^N}|u^{\tau_n}(x,t)-u^{\tau_k}(x,t)|dx\bigg]&\le\,\mathbb{E}\int_{(\mathbb{T}^N)^2}\theta_\epsilon (x-y)|u_0(x)-u_0(y)|dx dy\\& \qquad + C_T\bigg(\delta+\delta\,\epsilon^{-1}+\delta^{2\gamma}\,\epsilon^{-2}+\delta^{-1}\epsilon^2+h(\delta)+(\sqrt{\tau_n}-\sqrt{\tau_k})^2\,\epsilon^{-2}\bigg)
\end{align*}
Choose $\delta=\epsilon^\beta$ with $\beta\in(\frac{1}{\gamma},2)$. then we conclude, for all $t\in[0,T],$
\begin{align*}
	\mathbb{E}[\|u^{\tau_n}(t)-u^{\tau_k}(t)\|_{L^1(\mathbb{T}^N)}]\,\to\,0\,\,\,\,\,\,\,\,\text{as}\,\,\,\,\,\,\tau_n,\tau_k\,\to\,0.
\end{align*}
It shows that sequence of viscous solutions $u^{\tau_n}$ is cauchy in $L_{\mathcal{P}_T}^1(\Omega\times[0,T]\times\mathbb{T}^N)$. Therefore $u^{\tau_n}$ converges to $u\in L_{\mathcal{P}_T}^1(\Omega\times[0,T]\times\mathbb{T}^N)$ as $\tau_n\to 0$.

\subsection{Approximate kinetic solutions} Here we apply the techique of Appendix \ref{A} to derive the kinetic formulation to \eqref{4.1} that satisfied by $f^{\tau_n}(t)=\mathbbm{1}_{u^{\tau_n}(t)\,\textgreater\,\zeta}$ in the sense of $D'(\mathbb{T}^N\times\mathbb{R})$. It reads as follows:
\begin{align*}
	df^{\tau_n}(t)+F'^{\,\tau_n}\cdot\nabla f^{\tau_n}(t) dt -A:D^2f^{\tau_n}(t) dt- &\tau_n \Delta f^{\tau_n}(t) dt+ g_x^\lambda(f^{\tau_n}(t))dt\\
	&= \delta_{u^{\tau_n}(t)=\zeta} \Phi\, dW + \partial_{\zeta}(m^{\tau_n} -\frac{1}{2}\beta^2 \delta_{u^{\tau_n}(t)=\zeta}) dt, 
\end{align*}
where
\begin{align*}
	dm^{\tau_n}(x,t,\zeta)&=dm_1^{\tau_n}(x,t,\zeta)+d\eta_1^{\tau_n}(x,t,\zeta)+d\eta^{\tau_n}_2(x,t,\zeta),\\
dm_1^{\tau_n}(x,t,\zeta)&=\tau_n |\nabla u^{\tau_n}|^2\delta_{u^{\tau_n}=\zeta},dxdt\\
d\eta_1^{\tau_n}(x,t,\zeta)&=\int_{\mathbb{R}^N}|u^{\tau_n}(x+z)-\zeta|\mathbbm{1}_{Conv\{u^{\tau_n}(x),u^{\tau_n}(x+z)\}}(\zeta)\mu(z)dzd\zeta dxdt\\
d\eta^{\tau_n}_2(x,t,\zeta)&=|div\int_0^{u^{\tau_n}} \sigma(\zeta)d\zeta|^2 \delta_{u^{\tau_n}(x,t)}(\zeta)dxdt.
\end{align*}
It implies that for all $\varphi\in C_c^2(\mathbb{T}^N\times\mathbb{R}),$ $\mathbb{P}$-almost surely, for all $t\in[0,T]$,
\begin{align}\label{approximate formulation}
	\langle f^{\tau_n}(t),\varphi\rangle &= \langle f_0^n,\varphi \rangle +\int_0^t \langle f^{\tau_n}(s), F'^{\,\tau_n}(\zeta)\cdot\nabla_x \varphi \rangle ds +\int_0^t\langle f(s),A:D^2\varphi\rangle ds -\int_0^t \langle f^{\tau_n}(s), g_x^\lambda(\varphi)\rangle ds\notag\\
	&\qquad+ \int_0^t \int_{\mathbb{T}^N}\int_{\mathbb{R}} \beta_k(x,\zeta) \varphi(x,\zeta) d\mathcal{V}_{x,s}^n(\zeta) dx dw_k(s)+ \tau_n\int_0^t \langle f^{\tau_n}(s),\Delta \varphi \rangle ds.\notag \\
	& \qquad+\frac{1}{2}\int_0^t\int_{\mathbb{T}^N}\int_{\mathbb{R}} \beta^2(x,\zeta) \partial_{\zeta}\varphi(x,\zeta) d\mathcal{V}_{x,s}^n(\zeta)dx ds- m^{\tau_n}(\partial_{\zeta}\varphi)([0,t]).
\end{align}
\begin{proposition}\label{p estimate}
	For all $p\in [2,\infty)$, the viscous solution $u^{\tau_n}$ to \eqref{4.1} satisfies the following estimates:
	\begin{align}\label{lp estimate}
		\mathbb{E}\sup_{0\le t\le T}\|u^{\tau_n}(t)\|_{L^p(\mathbb{T}^N)}^p\le C (1+\mathbb{E}\|u_0\|_{L^p(\mathbb{T}^N)}^p),
	\end{align}
	\begin{align}\label{4.7}
		\mathbb{E}\sup_{0\le t\le T}\|u^{\tau_n}(t)\|_{L^p(\mathbb{T}^N)}^p+p(p-1)\int_0^T\int_{\mathbb{T}^N}\int_{\mathbb{R}} |\zeta|^{p-2} dm^{\tau_n}(x,s,\zeta)\le C (1+\mathbb{E}\|u_0\|_{L^p(\mathbb{T}^N)}^p),
	\end{align}
	where the constant C does not depend on $\tau_n$.
\end{proposition}
\begin{proof}
	In this proof, we use the generalized It\^o formula \cite[Appendix A]{vovelle}. As the generalized It\^o formula can not be applied directly to $\phi(\zeta)=|\zeta|^p$, $p\in [2,\infty)$, and $\kappa(x)=1$. Here we follow the approach of \cite{denis,vovelle} and introduce functions $\phi_l\in C^2(\mathbb{R})$ that approximate $\phi$ and have quadratic growth at infinity as required by Proposition \cite[Appendix A]{vovelle}. We define
	\[\phi_l(\zeta)=\begin{cases}
		|\zeta|^p, & |\zeta|\,\le\, l,\\
		l^{p-2}\bigg[\frac{p(p-1)}{2}\zeta^2-p(p-2)l|\zeta|+\frac{(p-1)(p-2)}{2}l^2\bigg], & |\zeta|\,\textgreater \,l.
	\end{cases}
	\]
	It is now easy to see that
	$$|\zeta \phi_l'(\zeta)| \,\le \, \phi_l (\zeta),$$
	$$|\phi_l'(\zeta)| \, \le p(1+\phi_l(\zeta)),$$
	$$|\phi_l'(\zeta)| \, \le |\zeta|\phi''(\zeta),$$
	$$\zeta^2 \phi_l ''(\zeta)\, \le p(p-1) \phi_l (\zeta),$$
	$$\phi_l''(\zeta) \, \le p(p-1)(1+\phi_l(\zeta))$$
	holds true for all $\zeta\in\mathbb{R}$, $l\in\mathbb{N}$, $p\in[2,\infty)$, then by  generalized It\^o formula \cite[Appendix A]{vovelle} we have for all $t\in[0,T]$
	\begin{align*}
		\int_{\mathbb{T}^N} \phi_l(u^{\tau_n}(t)) dx &= \int_{\mathbb{T}^N} \phi_l(u_0) dx -\int_0^t \int_{\mathbb{T}^N} \phi_l'(u^{\tau_n}) \mbox{div}(F^{\tau_n}(u^{\tau_n})) dx ds-\int_0^t\int_{\mathbb{T}^N} \phi_l'(u^{\tau_n}) g_x^\lambda(u^{\tau_n}) dx ds\\&\qquad+\int_0^t\int_{\mathbb{T}^N} \phi_l'(u^{\tau_n}) div(A(u^{\tau_n})\nabla u^{\tau_n}) dx ds+\int_0^t\int_{\mathbb{T}^N} \phi_l'(u^{\tau_n}) \tau_n \Delta u^{\tau_n} dx ds\\&\qquad+\sum_{k\ge1}\int_0^t\int_{\mathbb{T}^N} \phi_l'(u^{\tau_n}) \beta_k(u^{\tau_n},x) dx dw_k(s)+\frac{1}{2}\int_0^t\int_{\mathbb{T}^N} \phi_l''(u^{\tau_n}) \beta^2(u^{\tau_n}) dx ds.  
	\end{align*}
	Let us define as $H^n(\zeta)=\int_0^\zeta \phi_l''(\zeta) F^{\tau_n}(\zeta) d\zeta$, then it is easy to conclude that the second term on the right-hand side vanishes due to the boundary conditions. The third, fourth and fifth terms are nonpositive. Indeed,
	\begin{align*}\int_0^t\int_{\mathbb{T}^N} \phi_l'(u^{\tau_n}) \mbox{div} (A(u^{\tau_n})\nabla u^{\tau_n}) dx ds &= - \int_0^t \int_{\mathbb{T}^N} \phi_l''(u^{\tau_n})|\sigma(u^{\tau_n})\nabla u^{\tau_n}|^2 dx ds,\\
	-\int_0^t\int_{\mathbb{T}^N}\phi_l'(u^{\tau_n}) g_x^\lambda(u^{\tau_n}) dx ds&=-\int_0^t\int_{\mathbb{T}^N}\int_{\mathbb{R}}\phi_l''(\zeta)\eta_1^{\tau_n}(x,\zeta,s),
	\\
	\int_0^t\int_{\mathbb{T}^N} \phi_l'(u^{\tau_n}) \tau_n \Delta u^{\tau_n} dx ds&= -\int_0^t\int_{\mathbb{T}^N}\phi_l''(u^{\tau_n}) \tau_n |\nabla u^{\tau_n}|^2 dx ds.
	\end{align*}
	After onwards we can follow proof of \cite[Proposition 4.3]{vovelle} to conclude result.
\end{proof}
\begin{remark}[\textbf{Uniform bound in $H^{2\lambda}$}]
	We have uniform $H^{2\lambda}$-bound of viscous solutions as consequence of proof of above proposition. In particular for p=2, we get
	\begin{align*}
			\frac{1}{2}\int_{\mathbb{T}^N} |u^{\tau_n}(t)|^2 dx &= \frac{1}{2}\int_{\mathbb{T}^N} |u_0|^2 dx -\int_0^t \int_{\mathbb{T}^N} u^{\tau_n} \mbox{div}(F^{\tau_n}(u^{\tau_n})) dx ds+\int_0^t\int_{\mathbb{T}^N} u^{\tau_n} div(A(u^{\tau_n})\nabla u^{\tau_n}) dx ds\\&\qquad-\int_0^t\int_{\mathbb{T}^N} u^{\tau_n} g_x^\lambda(u^{\tau_n}) dx ds+\int_0^t\int_{\mathbb{T}^N} u^{\tau_n} \tau_n \Delta u^{\tau_n} dx ds\\&\qquad+\sum_{k\ge1}\int_0^t\int_{\mathbb{T}^N} u^{\tau_n} \beta_k(u^{\tau_n},x) dx dw_k(s)+\frac{1}{2}\int_0^t\int_{\mathbb{T}^N} \beta^2(x, u^{\tau_n}) dx ds.  
	\end{align*}
A simple consequence of BDG inequality, Gronwall's inequality implies that, for all $t\in[0,T]$
	\begin{align*}
		\mathbb{E}\int_{\mathbb{T}^N} |u^{\tau_n}(t)|^2 dx + \mathbb{E}\int_0^T\int_{\mathbb{T}^N}|g_x^{\frac{\lambda}{2}}(u(x))|^2\dx
		\le\,C\,\mathbb{E}\int_{\mathbb{T}^N}|u_0(x)|^2\dx.
	\end{align*}
It shows that $u^{\tau_n}$ uniformly bounded in $L^2(\Omega;L^2([0,T];H^{\lambda}(\mathbb{T}^3))).$	
\end{remark}
\subsection{Convergence of approximate kinetic functions}\label{section 5.3}In this subsection, we shall use the following notions: we say that a sequence $(\mathcal{V}^{\tau_n})$ of Young measures converges to $\mathcal{V}$ in $\mathcal{Y}^1$ if \eqref{2.12} is satisfied. A random Young measure is a $\mathcal{Y}^1$- valued random variable  by definition. We define Young measures on $\mathbb{T}^N\times[0,T]$ as following
$\mathcal{V}^{\tau_n}=\delta_{u^{\tau_n}=\zeta}, $
${\mathcal{V}}=\delta_{{u}=\zeta}$
and sequence $(\mathcal{V}^{\tau_n})$ of Young measures satisfies 
\begin{align}\label{4.10}\mathbb{E}\Bigg[\sup_{t\in[0,T]}\int_{\mathbb{T}^N}\int_{\mathbb{R}}|\zeta|^p d\mathcal{V}_{x,t}^{\tau_n}(\zeta) dx\Bigg]\le\,C_p.
\end{align}
\begin{proposition}\label{D}
	It holds true $($up to subsequences$)$
	\begin{enumerate}
		\item[1.] sequence of Young measures ${\mathcal{V}}^n$ converge to ${\mathcal{V}}$ in $\mathcal{Y}^1$,  ${\mathbb{P}}$-a.s.
		\item[2.] ${\mathcal{V}}$ satisfies
		\begin{align}\label{4.11}
			{\mathbb{E}}\Bigg(\sup_{K\subset[0,T]}\frac{1}{|K|}\int_{K}\int_{\mathbb{T}^N}\int_{\mathbb{R}}|\zeta|^p d{\mathcal{V}}_{x,t}(\zeta)dx dt\Bigg)&\le\, C_p,
		\end{align}
		where the supremum in \eqref{4.11} is a countable supremum over all open intervals $K\subset[0,T]$ with rational end points. Furthermore, if ${f}^{\tau_n}, {f}:\mathbb{T}^N\times[0,T]\times\mathbb{R}\times{\Omega}\to [0,1]$ are defined by
		$${f}^{\tau_n}(x,t,\zeta)={\mathcal{V}}_{x,t}^{\tau_n}(\zeta,+\infty),\,\,\,\, {f}(x,t,\zeta)={\mathcal{V}}_{x,t}(\zeta,+\infty),$$
		then ${f}^{\tau_n}\to{f}$ in $L^{\infty}(\mathbb{T}^N\times[0,T]\times\mathbb{R})$-weak-* ${\mathbb{P}}$-almost surely.
		\item[3.] There exist a full measure suset $D$ containing $0$ of $[0,T]$ such that for all $t\in D$
		$${f}^{\tau_n}\to {f}\,\,\, in\,\,\,\,\,\,L^\infty (\Omega\times\mathbb{T}^N\times\mathbb{R}).$$
		$$$$
	\end{enumerate}
\end{proposition}
\begin{proof} Markov inequality and strong convergence of $u^{\tau_n}$ in $L^1(\Omega\times[0,T]\times\mathbb{T}^N)$, implies that $\p$-a.s, $u^{\tau_n}$ (upto subsequecnce) converges $u$ in $L^1([0,T]\times\mathbb{T}^N).$ It implies that, ${\mathbb{P}}$-almost surely, sequence of Young measures ${\mathcal{V}^{\tau_n}}$ converges to ${\mathcal{V}}$ in $\mathcal{Y}^1$.
	\noindent
	For second point, we follow proof of \cite[Proposition 4.3]{sylvain}. Since the map
	$$\kappa_p:\mathcal{Y}^1\to[0,+\infty],\,\,\, \mathcal{V}\mapsto\,\, \sup_{J\subset[0,T]}\frac{1}{|J|}\int_{J}\int_{\mathbb{T}^N}\int_{\mathbb{R}}|\zeta|^p\, d\mathcal{V}_{x,t}(\zeta)\,dx\,dt$$
	is lower semi-continuous, we have
	$${\mathbb{E}}\, \kappa_p({\mathcal{V}})\le\liminf_{\tau_n\to 0}{\mathbb{E}}\,\kappa_p({\mathcal{V}}^{\tau_n})\le\,C_p$$ 
	Consequently, the ${\mathcal{V}}$ satisfies the condition
	$${\mathbb{E}}\bigg(\sup_{J\subset[0,T]}\frac{1}{|J|}\int_{J}\int_{\mathbb{T}^N}\int_{\mathbb{R}}|\zeta|^p\,d\mathcal{V}_{x,t}(\zeta)\,dx\,dt\bigg)\le\,C_p.$$
	If we apply Corollary \ref{th2.7}, we obtain that ${f}^{\tau_n}\to{f}$ in $L^\infty(\mathbb{T}^N\times[0,T]\times\mathbb{R})$-weak-* ${\mathbb{P}}$-almost surely.
	By using similar argument, there exist  full measure subset $D$ of $[0,T]$ such that for all $t\in D$
	$${f}^{\tau_n}\to {f} \,\,\, in\,\,\,\,\, L^\infty(\Omega\times\mathbb{T}^N\times\mathbb{R}).$$
\end{proof}
\subsection{Limit of the random measures}\label{section 5.4}
	Suppose $\mathcal{M}_b(\mathbb{T}^N\times[0,T]\times{R})$ is denote the space of bounded Borel measures on $\mathbb{T}^N\times[0,T]\times\mathbb{R}$ equipped norm is given by the total variation of measures. This space is the dual space to the space of all continuous functions vanishing at infinity $C_0(\mathbb{T}^N\times[0,T]\times\mathbb{R})$ equipped with the supremum norm. It is separable, so the following duality holds for $r$, $r^*\in(1,\infty)$ being conjugate exponents 
$$(L^{r^*}({\Omega};C_0(\mathbb{T}^N\times[0,T]\times\mathbb{R}))^*\simeq L_w^r({\Omega};\mathcal{M}_b(\mathbb{T}^N\times[0,T]\times\mathbb{R}),$$
where the space $L_w^r({\Omega};\mathcal{M}_b(\mathbb{T}^N\times[0,T]\times\mathbb{R})$ contains all $weak^*$-measurable mappings $\eta:{\Omega}\to\mathcal{M}_b(\mathbb{T}^N\times[0,T]\times\mathbb{R})$ such that
$${\mathbb{E}}\left\Vert \eta\right\Vert_{\mathcal{M}_b}^q\textless\,\infty.$$ 
\begin{lem}\label{m}
	There exists a kinetic measure ${m}$ such that
	\[ m^{\tau_n}\overset{w^*}{\to}{m}\,\,\qquad in\,\,\qquad\qquad L_w^2({\Omega}; \mathcal{M}_b(\mathbb{T}^N\times[0,T]\times\mathbb{R}))-weak^* \]
	Moreover, ${m}$ can be rewritten as ${\eta}_1+\eta_2+{m}_1$, where
	\begin{align*} d{\eta}_1(x,t,\zeta)&=\int_{\mathbb{R}^N}|u(x+z,t)-\zeta|\mathbbm{1}_{Conv\{{u}(x+z,t),{u}(x,t)\}}(\zeta)\mu(z)dz dx d\zeta dt,\\
	d\eta_2(x,t,\zeta)&=|\mbox{div}\int_0^u \sigma(\zeta)d\zeta|^2 \delta_{u(x,t)}(\zeta)dxdt,
	\end{align*}
	and $m_1$ is almost surely a nonnegative measure over $\mathbb{T}^N\times[0,T]\times\mathbb{R}$.
	
\end{lem}
\begin{proof}
	We observe that due to the computations used in the proof of \ref{p estimate}, it holds $\mathbb{P}$-almost surely,
	\begin{align*}
		&\int_0^T\int_{\mathbb{T}^N}\int_{\mathbb{R}}m^{\tau_n}(x,t,\zeta)\,d\zeta\,dx\,dt\\
		&\qquad\qquad\le\,C\, \left\Vert u_0\right\Vert_{L^2(\mathbb{T}^N)}^2
		+\,C\,\sum_{k\ge1}\int_0^T\int_{\mathbb{T}^N}\,u^{\tau_n}\,\beta_k(x, u^{\tau_n})\,dx\,dw_k(t)+\,C\,\int_0^T\int_{\mathbb{T}^N}\, \beta^2(x, u^{\tau_n})\,dx ds.
	\end{align*}
	Taking square and expectation and finally by using of the Ito isometry, we deduce
	\begin{align*}
		\mathbb{E}|m^{\tau_n}([0,T]\times\mathbb{T}^N\times\mathbb{R})|^2
		&=\mathbb{E}\bigg|\int_0^T\int_{\mathbb{T}^N}\tau_n |\nabla u^{\tau_n}|^2\,dx\,dt+\int_0^T\int_{\mathbb{T}^N}\int_{\mathbb{R}}\eta_1^{\tau_n}(x,t,\zeta)+\eta_2^n(x,t,\zeta)\,d\zeta\,dx\,dt\bigg|^2\,\\
		&\le\,C.
	\end{align*}
We obtain that
	$${\mathbb{E}}|m^{\tau_n}([0,T]\times\mathbb{T}^N\times{R})|^2\,\le\,C.$$
	therefore, the set $\{m^{\tau_n};n\in\mathbb{N}\}$ is bounded in $L_w^2({\Omega};\mathcal{M}_b(\mathbb{T}^N\times[0,T]\times\mathbb{R})$ and, making use of the Banach-Alaoglu theorem, it possesses a $weak^*$ convergent subsequence, still denoted by $\{m^{\tau_n},\;n\in\mathbb{N}\}$. It shows that there is ${\mathbb{P}}$-almost sure convergence of $(m^{\tau_n})$ to $m$ in $\mathcal{M}_b(\mathbb{T}^N\times[0,T]\times\mathbb{R})$-weak-*.
	Now, we want to show that $weak^*$ limit ${m}$ is actually a kinetic measure. The first point about measurablilty of $m$ in Definition \ref{kinetic measure} is straightforward.  For second point, we observe that due to \eqref{4.7},
	\begin{align}\label{4.13}
		\mathbb{E}\sup_{t\in[0,T]}\left\Vert u^{\tau_n}\right\Vert_{L^p(\mathbb{T}^N)^p}^p+p(p-1)\mathbb{E}\int_0^T\int_{\mathbb{T}^N}\int_{\mathbb{R}} |\zeta|^{p-2}&dm^{\tau_n}(x,t,\zeta)\le\,C(1+\mathbb{E}\left\Vert u_0\right\Vert_{L^p(\mathbb{T}^N)}^p).
	\end{align}
	Let $(\chi_{\delta})$ be a truncation on $\mathbb{R}$, then it holds, for $p\in[2,\infty),$
	\begin{align*}
		&\mathbb{E}\int_{[0,T]\times\mathbb{T}^N\times{R}}|\zeta|^{p-2}d{m}(x,t,\zeta)\le\,\liminf_{\delta\to0}\mathbb{E}\int_{[0,T]\times\mathbb{T}^N\times\mathbb{R}}|\zeta|^{p-2}\chi_{\delta}(\zeta) d{m}(t,x,\zeta)\\
		&=\liminf_{\delta\to 0}\lim_{\tau_n\to\,0}\mathbb{E}\int_{[0,T]\times\mathbb{T}^N\times{R}}|\zeta|^{p-2}\chi_{\delta}(\zeta)dm^{\tau_n}(t,x,\zeta)\le\,C,
	\end{align*}
	where the last inequality follows from \eqref{4.13}. It shows that $m$ vanishes for large $\zeta$.
	
	\noindent
	Finally, by the same approach as above, we deduce that there exist measures $o_1,\, o_2$ and $o_3$ such that
	$$m_1^{\tau_n}\overset{w^*}\to o_1,\,\,\,\,\,\eta_1^{\tau_n}\overset{w^*}\to\,o_2\,\,\,\,\,\eta_2^{\tau_n}\overset{w^*}\to\,o_3\,\,\,\,\, in\,\,L_w^2({\Omega};\mathcal{M}_b(\mathbb{T}^N\times[0,T]\times\mathbb{R}) $$
	Then from \eqref{4.7} we obtain
	$${\mathbb{E}}\int_0^T\int_{\mathbb{T}^N}\bigg|\mbox{div}\int_0^{u^{\tau_n}}\sigma(\zeta)\,d\zeta\bigg|^2 \,dx\,dt\le\,C.$$
	Therefore as an application of the Banach-Alaoglu theorem we get that, up to subsequence, $\mbox{div}\int_0^{u^{\tau_n}}\sigma(\zeta)\,d\zeta$ converges weakly in $L^2({\Omega}\times[0,T]\times\mathbb{T}^N)$. On the other hand, from the strong convergence  and the fact that $\sigma\in C_b(\mathbb{R})$,  using integration by parts, we conclude for all $\varphi\in C^1([0,T]\times\mathbb{T}^N),\,\mathbb{P}-a.s.,$
	$$\int_0^T\int_{\mathbb{T}^N}\bigg(\mbox{div}\int_0^{u^{\tau_n}}\sigma(\zeta) d\zeta\bigg)\varphi(t,x) dx dt\to \int_0^T\int_{\mathbb{T}^N}\bigg(\mbox{div}\int_0^{{u}}\sigma(\zeta)\,d\zeta\bigg)\varphi(t,x)\,dx\,dt,$$
	and therefore ${\mathbb{P}}$-a.s.
	\begin{align}\label{4.14}
		\mbox{div}\int_0^{{u}^{\tau_n}}\sigma(\zeta)d\zeta\overset{w}\to\,\mbox{div}\int_0^{{u}}\sigma(\zeta)\,,d\zeta\,\,\,\,\, in\,\, L^2([0,T]\times\mathbb{T}^N).
	\end{align}
	Since any norm is weakly lower semicontinuous, it implies that for all $\phi\in C_0([0,T]\times\mathbb{T}^N\times{\mathbb{R}})$ and fixed $\zeta\in\mathbb{R}$, $\mathbb{P}$-a.s.
	\begin{align*}
		&\int_0^T\int_{\mathbb{T}^N}\bigg|\mbox{div}\int_0^{{u}}\sigma(\zeta)d\zeta\bigg|^2 \phi^2(t,x,\zeta)\,dx\,dt\le\,\liminf_{n\to\infty}\int_0^T\int_{\mathbb{T}^N}\bigg|\mbox{div}\int_0^{u^{\tau_n}}\sigma(\zeta)\,d\zeta\bigg|^2\phi^2(t,x,\zeta)\,dx\,dt
	\end{align*}
	and by the Fatou's lemma, we have ${\mathbb{P}}$-a.s.
	\begin{align*}
		&\int_0^T\int_{\mathbb{T}^N}\int_{\mathbb{R}}\bigg|\mbox{div}\int_0^{{u}}\sigma(\zeta)\,d\zeta\bigg|^2\phi^2(t,x,\zeta)\,d\delta_{{u}=\zeta}\,dx\,dt\\
		&\qquad\qquad\qquad\qquad\,\le\, \liminf_{n\to \infty}\int_0^T\int_{\mathbb{T}^N}\int_{\mathbb{R}}\bigg|\mbox{div}\int_0^{u^{\tau_n}}\sigma(\zeta)\,d\zeta\bigg|^2\phi^2(t,x,\zeta)\,d\delta_{u^{\tau_n}=\zeta}\,dx\,dt
	\end{align*}
It shows that $\p$-a.s., ${\eta}_2\le\,o_3$.
	Recall that, upto subsequece, $\p$-almost surely, $u^{\tau_n}(x,t)\to{u}(x,t)$ for $(t,x)\notin N$ with $N\subset\,\mathbb{T}^N\times[0,T]$ negligible subset. $\mathbb{P}$-almost surely, fixing $z\in\mathbb{R}^N$, we thus have also $u^{\tau_n}(x+z,t)\to{u}(x+z,t)$ for any $(x,t)$ not in some negligible subset $N_z$ of $\mathbb{T}^N\times[0,T]$.
	Hence we have
	$$|u^{\tau_n}(x+z,t))-\zeta|\mathbbm{1}_{Conv\{u^{\tau_n}(x,t),u^{\tau_n}(x+z,t)\}}(\zeta)\to|u(x+z,t))-\zeta|\mathbbm{1}_{Conv\{{u}(x,t),{u}(x+z,t)\}}(\zeta)\,\,\, $$
	$as\,\,\, n\to \infty,$
	for any $(x,t,\zeta)\in\mathbb{T}^N\times[0,T]\times\mathbb{R}$ such that $(x,t)\notin N\cup N_z$ and $\zeta\ne {u}(t,x)$. Note that latter subset of $\mathbb{T}^N\times[0,T]\times\mathbb{R}$, on which previous convergence does not hold, has Lebesgue measure zero. We can then use Fatou's lemma to conclude $\p$-a.s., ${\eta}_1\,\le\,o_2$.
	Hence $${m}=o_1+o_2\ge \,{\eta}_1+\eta_2\,\,\,\,\,\,\,\mathbb{P}-\text{almost surely}.$$
	Concerning the chain rule formula \eqref{chain}, due the regularity of $u^{\tau_n}$, we observe that, for all $u^{\tau_n}$ and for any $\varphi\in C_b(\mathbb{R})$,
	\begin{align}\label{ps}
		\mbox{div}\int_0^{u^{\tau_n}} \varphi(\zeta) \sigma(\zeta) d\zeta=\varphi(u^{\tau_n})\mbox{div}\int_0^{u^{\tau_n}}\sigma(\zeta) d\zeta\,\qquad\,\,\,\text{in}\,\qquad\,\mathcal{D}'(\mathbb{T}^N),\,\,\,\text{a.e.}\,\,(\omega,t).
	\end{align}
	Furthermore, it is easy to obtain \eqref{4.14} with the integrant $\sigma$ replaced by $\varphi\,\sigma$, we can pass to the limit on the left hand side and making use of the strong-weak convergence, also on the right hand side of \eqref{ps}. The proof is complete.
\end{proof}
\subsection{Kinetic solution}\label{section 5.5}
As a consequence of convergence results, stated in previous subsections and as proposed in \cite[Lemma 2.1, Propostion 4.9]{sylvain}, we can pass to the limit in all the terms of \eqref{approximate formulation}. Convergence of the stochastic integral can be verified easily using strong convergence and uniforn $L^p$-bound of $u^{\tau_n}$.

\noindent
For all $\varphi\in C_c^2(\mathbb{T}^N\times\mathbb{R}),$ for ${\mathbb{P}}$-allmost surely, there exists a negligible set $N_0\subset[0,T]$ such that for all $ t\in[0,T]\backslash N_0$,
\begin{align}\label{4.17}
	\langle {f}(t),\varphi \rangle &= \langle {f}_0, \varphi \rangle + \int _0^t \langle {f}(s), F'(\zeta)\cdot\nabla\varphi \rangle ds+\int_0^t\langle f(s),A:D^2\varphi\rangle ds- \int_0^t \langle {f}(s), g_x^\lambda[\varphi] \rangle ds\notag\\
	&\qquad+\sum_{k=1}^\infty \int_0^t\int_{\mathbb{T}^N}\int_{\mathbb{R}} \beta_k(x,\zeta) \varphi (x,\zeta) d{\mathcal{V}}_{x,s}(\zeta) dx d w_k(s)\notag\\
	&\qquad +\frac{1}{2} \int_0^t \int_{\mathbb{T}^N}\int_{\mathbb{R}} \partial_{\zeta} \varphi( x,\zeta) \beta^2(x,\zeta) d{\mathcal{V}}_{x,s}(\zeta) dx ds  - {m}(\partial_{\zeta}\varphi)([0,t]),
\end{align}
 We now want to show that the above formulation holds for all  time $t\in[0,T]$. For that, we use the following Proposition.
\begin{proposition}\label{th4.15}
	There exists a measurable subset ${\Omega}^+$ of ${\Omega}$ of probability one, a random Young measure ${\mathcal{V}}^+$ on $\mathbb{T}^N\times(0, T)$ such that
	\begin{enumerate}
		\item[1.] for all ${\omega}\in{\Omega}^+$, for almost every $(x,t)\in\mathbb{T}^N\times(0,T)$, the probability measures ${\mathcal{V}}_{x,t}^+$  and ${\mathcal{V}}_{x,t}$ coincide.
		\item[2.] the kinetic function ${f}^+(x,t,\zeta):= {\mathcal{V}}_{x,t}^+(\zeta,+\infty)$ satisfies: for all ${\omega}\in{\Omega}^+$, for all $\varphi\in C_c(\mathbb{T}^N\times(0,T))$, $t\mapsto \langle {f}^+(t), \varphi \rangle$ is c\'adl\'ag.
		\item[3.] The random Young measure ${\mathcal{V}}^+ $ satisfies
		$${\mathbb{E}}\sup_{t\in[0,T]}\int_{\mathbb{T}^N}\int_{\mathbb{R}}|\zeta|^p d{\mathcal{V}}_{x,t}^+(\zeta) dx\,\,\le\,\, C_p.$$
	\end{enumerate}
\end{proposition}
\begin{proof}
	For proof, we refer to \cite[Proposition 4.8]{sylvain}.
\end{proof}
We will now consider only the c\'adl\'ag versions. We replace ${\mathcal{V}}$ by ${\mathcal{V}}^+$, ${f}$ by ${f}^+$ (using of It\^o isometry) and making use of \cite[Lemma 4.14]{sylvain} for measurability issue. Let us recall the fact that \eqref{4.17} is now true for all t. For all $t\in[0,T]$, for all $\varphi\in C_c^2(\mathbb{T}^N\times\mathbb{R})$ 
\begin{align}\label{4.26}
	\langle {f}^+(t),\varphi \rangle &= \langle {f}_0, \varphi \rangle + \int _0^t \langle {f}^+(s), F'(\zeta)\cdot\nabla\varphi \rangle ds +\int_0^t\langle f^{+}(s),A:D^2\varphi\rangle ds- \int_0^t \langle {f}^+(s), g_x^\lambda[\varphi] \rangle ds\notag\\
	&\qquad +\sum_{k=1}^\infty \int_0^t\int_{\mathbb{R}}\int_{\mathbb{T}^N} \beta_k(x,\zeta) \varphi (x,\zeta) d{\mathcal{V}}_{x,s}^+(\zeta) dx dw_k(s)\notag\\
	&\qquad +\frac{1}{2} \int_0^t\int_{\mathbb{R}} \int_{\mathbb{T}^N} \partial_{\zeta} \varphi( x,\zeta ) \beta^2(x,\zeta) d{\mathcal{V}}_{x,s}^+(\zeta) dx ds  - {m}(\partial_{\zeta}\varphi)([0,t])
\end{align}
$\p$-a.s..
Here we are not interested in this form of solution, ${f}^+$ is kinetic function. But, it should be equlibrium for all $t\in[0,T]$. In the case of stochastic conservation laws, the authors prove the existence of solution from reduction of generalized solution (cf.\cite[Proposition 2.8\,$\&$\,Theorem 3.2]{sylvain}). Here we also have similar circumstances. We can follow similar approach to get required form. Indeed, we have  ${\mathbb{P}}$-almost surely, almost $t\in[0,T]$,   ${f}^+(t)=\mathbbm{1}_{{u}(t)\textgreater\zeta}$ for all $\zeta\in\mathbb{R}$. For getting equlibrium form of ${f}^+$ for all $t\in[0,T]$, we have to repeat proof of Theorem \ref{th3.6} for ${f}^+$ with $m\ge\,\eta_1+\eta_2$. After that  we have, $\p$-almost surely, for all $t\in[0,T]$
$${f}^+(t)=\mathbbm{1}_{\tilde{u}(t)\textgreater\zeta}$$
where 
$$\tilde{u}(t)=\int_{\mathbb{R}} ({f}^+(t)-\mathbbm{1}_{0\textgreater\zeta})\,\,d\zeta$$
Since ${\mathbb{P}}$-almost surely, $u(x,t)=\tilde{u}(x,t)$ almost $(x,t)\in\mathbb{T}^N\times[0,T]$, therefore ${(\tilde{u}(t))_{t\in[0,T]}}$ and $f^+(t)=\mathbbm{1}_{\tilde{u}(t)\textgreater\zeta}$, satisfy the all points of Definition \ref{kinetic solution} of kinetic solution. It shows that $(\tilde{u}(t))_{t\in[0,T]}$ is kinetic solution.
\subsection{Existence-general initial data}
In this final subsection, we prove an existence proof in the genral case of $u_0\in L^p(\Omega; L^p(\mathbb{T}^N)),$ for all $p\in[1,+\infty)$. It is direct consequence of the previous subsections. We approximate the initial conditions by a sequence $\{u_0^\delta\}\subset\,L^p(\Omega; C^\infty(\mathbb{T}^N)),$ $p\in[1,+\infty)$ such that $u_0^\delta\,\to\,u_0$ in $L^1(\Omega;L^1(\mathbb{T}^N))$. That is, the initial conditions $u_0^\delta$ can be defined as a pathwise molification of $u_0$ so that holds true
\begin{align}\label{4.27}
	\|u_0^\delta\|_{L^p(\Omega; L^p(\mathbb{T}^N))}\,\le\,\|u_0\|_{L^p(\Omega; L^p(\mathbb{T}^N))}
\end{align}
According to the previous subsections, for each $\delta\in(0,1)$, there exists a kinetic solution $u^\delta$ to \eqref{1.1} with initial condtion $u_0^\delta$. As consequence of the contraction principle, for all $t\in[0,T]$
\[\mathbb{E}\|u^{\delta_1}(t)-u^{\delta_2}(t)\|_{L^1(\mathbb{T}^N)}\,\le\,\mathbb{E}\|u_0^{\delta_1}-u_0^{\delta_2}\|_{L^1(\mathbb{T}^N)}\,\,\,\qquad \delta_1, \delta_2 \in (0,1),\]
It shows that there exists $u\in L^p_{\mathcal{P}_T}(\Omega\times[0,T]\times\mathbb{T}^3)$ such that $u^{\delta}$ converges to $u$ in $L^1_{\mathcal{P}_T}(\Omega\times[0,T]\times\mathbb{T}^3)$ as $\delta\,\to\,0$.
By \eqref{4.27} and \eqref{lp estimate}, we still have spatial regularity and the uniform energy estimate, that is for $p\in[1,+\infty)$, 
	$$\mathbb{E}\, \sup_{0\le t\le T}\|u^\delta(t)\|_{L^p(\mathbb{T}^N)}^p\,\le\,C({T,u_0}),$$
	$$\mathbb{E}\big[\|u^\delta\|_{L^2(0,T;H^{\lambda}(\mathbb{T}^N))}^2\big]\,\le\,C(T,u_0),$$
as well as (using the usual notation)
\begin{align*}
	\mathbb{E}|m^\delta(\mathbb{T}^N\times[0,T]\times\mathbb{R})|^2\,\le\,C_{T,u_0}.
\end{align*}
We can also conclude that Lemma \ref{m} also true for sequence of kinetic measure $m^\delta$.
With these informations in hand, we are ready to pass the limit in \eqref{2.6} and conclude that there exists a kinetic solution $(u(t))_{t\in[0,T]}$ to \eqref{1.1}. The proof of Theorem \ref{th2.10} is thus complete.
\begin{remark}[Convergence of approximations at fixed time t]\label{fixed time}
	We define $$B_{at}^\omega=\{t\in[0,T]; m(\mathbb{T}^N\times\{ t \}\times\mathbb{R})\,\textgreater\,0\}.$$
	$\mathbb{P}$-almost
	surely, there is a countable subset $B_{at}^{\omega} \subset [0, T]$ such that for all $t\in[0, T]\setminus B_{at}^\omega$, for all $\varphi\in C_c(\mathbb{T}^N \times\mathbb{R})$, $\langle f^\varepsilon(t),\varphi\rangle \to  $$\langle f(t), \varphi\rangle$ ( cf. \cite[Proposition 4.9]{sylvain}). Making use of Proposition \ref{th2.8} gives, almost surely,
	for all $t\in [0, T ]\setminus B_{at}^\omega $, $u^\varepsilon (t)\to u(t)$ in $L^p(\mathbb{T}^N)$. We have $\p$-almost surely, for all $t\in[0,T]$, $f(t)=f^-(t)$ (cf.\cite[Proposition 2.11]{sylvain}). It follows that $B_{at}^\omega$ is empty. This gives that $\mathbb{P}$-almost surely, for all $t\in[0, T ]$, $u^\varepsilon (t) \to u(t)$ in $L^p(\mathbb{T}^N)$.
\end{remark}
\section{Continuous dependence estimate: proof of Theorem \ref{main theorem 2}}\label{section5}
In this section, we develop a general framework for the continuous dependence estimate by using BV estimate of kinetic solutions.
\subsection{BV estimate}
\begin{thm}
	Let $(u(t))_{t\in[0,T]}$ be a kinetic solution to \eqref{E.1} with initial data $u_0$. Then the following BV estimate holds: for all $t\in[0,T]$
	\begin{align}
		\mathbb{E}[\|u(t)\|_{BV}]\,\le\,\mathbb{E}[\|u_0\|_{BV}].
	\end{align}
\end{thm}
\begin{proof}
	By the substitution $z=x+h$ in \eqref{E.1}, it implies that that, if $u(x,t)$ solves \eqref{E.1} with initial data $u_0(x)$, then u(z,t)=u(x+h,t) solves \eqref{1.1} with initial data $u_0(x+h)$. From estimates given in the proof of Theorem $\ref{th3.6}$ we have for all $t\in[0,T]$,
	\begin{align*}
		&\mathbb{E}\int_{(\mathbb{T}^N)^2}\int_{\mathbb{R}^2}f(x,t,\xi)\bar{f}(y+h,t,\zeta)\kappa_{\delta}(\xi-\zeta)\theta_{\epsilon}(x-y)d\xi d\zeta dx dy\\
		&\le\mathbb{E}\int_{(\mathbb{T}^N)^2}\int_{\mathbb{R}^2}\theta_{\epsilon}(x-y)\kappa_{\delta}(\xi-\zeta)f_{0}(x,\xi)\bar{f}_{0}(y+h,\zeta)d\xi d\zeta dx dy+Ct\delta^{\lambda_{F_1}}\epsilon^{-1}+Ct\delta^{\lambda_{F_2}}+ Ct\epsilon^{-2}\,\delta^{2\gamma_a}.\notag
	\end{align*}
	taking $\delta\to 0$, then for all $t\in[0,T]$
	\begin{align*}
		&\mathbb{E}\int_{(\mathbb{T}^N)^2}(u(y+h,t)-u(x,t))_{+}\theta_\epsilon(x-y) dx dy\le\mathbb{E}\int_{(\mathbb{T}^N)^2}(u_0(y+h)-u_0(x))_{+}\theta_{\epsilon}(x-y) dx dy \notag\\
	\end{align*}	
	letting $\epsilon\to0$, then it implies that for all $t\in[0,T]$
	\begin{align*}
		\mathbb{E}\int_{\mathbb{T}^N}\frac{|u(x+h,t)-u(x,t)|}{|h|}dx\le\,\mathbb{E}\int_{\mathbb{T}^N}\frac{|u_0(x+h)-u_0(x)|}{|h|}dx\\
	\end{align*}
We conclude that  for all $t\in[0,T]$,
	$$\mathbb{E}[TV_x(u(\cdot,t))]\le\,\mathbb{E}[TV_x(u_0)].$$
	This finishes the proof.
\end{proof}
\noindent
Now we apply Kruzhkov's doubling of variable technique and try to bound the difference of kinetic solutions. Proof of following proposition is similar as Proposition \ref{th3.5}, so we left to reader.
\begin{proposition}
	Let $(u(t))_{t\in[0,T]}$ be a kinetic solution to \eqref{E.1} with initial data $u_0$, and let $(v(t))_{t\in[0,T]}$ be a kinetic solution to \eqref{E.2} with intial data $v_0$. Then, for all $t\in[0,T]$ and non-negative test functions $\theta\in{C}^\infty(\mathbb{T}^N)$,\, $\kappa\in{C_{c}^\infty(\mathbb{R})}$, we have 
	\begin{align}\label{ku}
		&\mathbb{E}\int_{(\mathbb{T}^N)^2}\int_{(\mathbb{R})^2}\theta(x-y)\kappa(\xi-\zeta)f_1(x,t,\xi)\bar {f}_2(y,t,\zeta)d\xi d\zeta dx dy\notag\\
		&\leq\mathbb{E}\bigg[\int_{(\mathbb{T}^N)}\int_{(\mathbb{R}^2)}\theta(x-y)\kappa(\xi-\zeta) f_{1,0} (x,\xi)\bar f_{2,0} (y,\zeta)d\xi d\zeta dx dy+ \mathcal{R}_{\theta}+ \mathcal{R}_{\kappa}+J+K\bigg], 
	\end{align}
	where
	\begin{align}
		\mathcal{R}_{\theta}\notag&=\int_0^t\int_{(\mathbb{T}^N)^2}\int_{\mathbb{R}^2}f_1(x,s,\xi)\bar{f}_2(y,s,\zeta)(F'(\xi)-G'(\zeta))\kappa(\xi-\zeta)d\xi d\zeta\cdot \nabla\theta(x-y)dx dy ds,\notag\\
		\mathcal{R}_{\kappa}\notag&=\frac{1}{2}\int_{(\mathbb{T}^N)^2}\theta(x-y)\int_0^t\int_{\mathbb{R}^2}\kappa(\xi-\zeta)|\Phi(\xi)-\Psi(\zeta)|^2d\mathcal{V}_{x,s}^{1}\oplus\mathcal{V}_{y,s}^{2}(\xi,\zeta)dx dy ds\notag\\
		J&=-\int_0^t\int_{(\mathbb{T}^N)^2}\int_{\mathbb{R}^2}f_1 (x,s,\xi)\bar {f}_2(y,s,\zeta)\kappa(\xi-\zeta)g_x^\lambda(\theta(x-y)))d\xi d\zeta dx dy ds\notag\\
		&\qquad-\int_0^t\int_{(\mathbb{T}^N)^2}\int_{\mathbb{R}^2}f_1 (x,s,\xi)\bar {f}_2(y,s,\zeta)\kappa(\xi-\zeta)g_y^\beta(\theta(x-y)))d\xi d\zeta dx dy ds\notag\\
		&\qquad-\int_0^t\int_{(\mathbb{T}^N)^2}\int_{\mathbb{R}^2}f_1(x,s,\xi)\partial_{\xi}\kappa(\xi-\zeta)\theta(x-y)d\eta_{v,2}(y,s,\zeta)dxd\xi \notag\\
		&\qquad+\int_0^t\int_{(\mathbb{T}^N)^2}\int_{\mathbb{R}^2}\bar{f}_2(y,s,\zeta)\partial_{\zeta}\kappa(\xi-\zeta)\theta(x-y)d\eta_{u,2}(x,s,\xi)dy d\zeta,\notag\\
		K&=\int_0^t\int_{(\mathbb{T}^N)^2}\int_{\mathbb{R}^2}f_1\bar{ f}_2(A(\xi)+B(\zeta)):D_x^2\theta(x-y)\kappa(\xi-\zeta)d\xi d\zeta dx dy ds\notag\\
		&\qquad-\int_0^t\int_{\mathbb{T}^N}^2\int_{\mathbb{R}^2}\theta(x-y)\kappa(\xi-\zeta) d\mathcal{V}_{x,s}^1(\xi)dx d\eta_{v,3}(y,s,\zeta)\notag\\
		&\qquad-\int_0^t\int_{\mathbb{T}^N}^2\int_{\mathbb{R}^2}\theta(x-y)\kappa(\xi-\zeta)d\mathcal{V}_{y,s}^2(\zeta)\,dy\,d\eta_{u,3}(x,s,\xi).\notag
	\end{align}
\end{proposition}
\subsection{\textbf{Proof of Theorem \ref{main theorem 2}} }Idea of this proof is similar to the proof of Theorem \ref{th3.6}. Here we define $\kappa_\delta$, $\theta_{\epsilon}$ and $\Phi_\delta$  similarly as in the proof of Theorem \ref{th3.6}. We will estimate each term of inequality \eqref{ku} separately in the following steps. 

\noindent
	\textbf{Step 1:} We estimate $\mathcal{R}_{\theta}$ as follows:
	\begin{align*}
		&|\mathbb{E}[\mathcal{R}_\theta]|=\bigg|\mathbb{E}\bigg[\int_0^t\int_{(\mathbb{T}^N)^2}\int_{\mathbb{R}^2}f_1(x,s,\xi)\bar{f}_2(y,s,\zeta)(F'(\xi)-G'(\xi)+G'(\xi)-G'(\zeta))\kappa_\delta(\xi-\zeta)d\xi d\zeta\cdot\\
		&\qquad\qquad\qquad\qquad\qquad\qquad\qquad\qquad\qquad\qquad\qquad\qquad\qquad\qquad\qquad\nabla\theta_\epsilon(x-y)dx dy ds\bigg]\bigg|\\
		&\le\,\bigg|\mathbb{E}\bigg[\int\limits_0^t\int\limits_{(\mathbb{T}^N)^2}\int\limits_{\mathbb{R}^2} f_1(x,s,\xi)\bar{f}_2(y,s,\zeta)\big(F'(\xi)-G'(\xi)\big)\kappa_\delta(\xi-\zeta)\cdot \nabla_x\theta_{\epsilon}(x-y)dxdyd\xi d\zeta ds\bigg]\bigg|+C\,t\, \epsilon^{-1}\,\delta^{\lambda_{G_1}}\\
		&\le\,\bigg|\mathbb{E}\bigg[\int\limits_0^t\int\limits_{(\mathbb{T}^N)^2}\int\limits_{\mathbb{R}^2} f_1(x,s,\xi)\nabla v\cdot\big(F(\xi)-G(\xi)\big) \theta_\epsilon(x-y) \kappa_{\delta}(\xi-\zeta) \delta_{\zeta=v} d\xi d\zeta dx dy ds\bigg]\bigg|+C\,t\epsilon^{-1}\,\delta^{\lambda_{G_1}}\\
		&\le\,\|F'-G'\|_{L^\infty(\mathbb{R})}\mathbb{E}\bigg[\int_0^t\int_{(\mathbb{T}^N)^2}|\nabla v| \theta_\epsilon(x-y) \Phi_{\delta}(u-v) dx dy ds\bigg]+C\,t\epsilon^{-1}\,\delta^{\lambda_{G_1}}\\
		&\le\,C\,t\,\mathbb{E}[\|v_0\|_{BV}]\|F'-G'\|_{L^\infty(\mathbb{R})}+ C\,t\,\epsilon^{-1}\,\delta^{\lambda_{G_1}}.
	\end{align*}
	\textbf{Step 2:} To improve the readability of the presentation in remaining paper, we make use of the following notation:
	$$
	g_x^\lambda[\varphi](x)= - \text{P.V.}\, \int_{\R^N} \big( \varphi(x+z) -\varphi(x)\big)\,d\mu_{\lambda}(z),  
	$$
	$$
	g_{x,r}^\lambda[\varphi](x)= -\int_{|z|\textless\,r} \big( \varphi(x+z) -\varphi(x)\big)\,d\mu_{\lambda}(z),  
	$$
	and 
	$$
	g_{x}^{\lambda,r}[\varphi](x)= -\int_{|z|\textgreater\,r} \big( \varphi(x+z) -\varphi(x)\big)\,d\mu_{\lambda}(z),  
	$$
	where $d\mu_{\lambda}(z):= \frac{dz}{|z|^{N + 2 \lambda}}$. Note that $\mu_{\lambda}$ is a nonnegative Radon measure on $\R^N\setminus \{0\}$ satisfying
	\begin{align}
		\label{important}
		\int_{\R^N\setminus \{0\}} \big(|z|^2 \wedge 1\big)\,d\mu_{\lambda}(z) < +\infty.
	\end{align}
For technical purposes (see \cite{jack,neeraj}), It is necessary to split Radon measures $\mu_{\lambda}, \mu_{\beta}$ as follows: Let $S^{\pm}$ be the sets such that
\begin{equation}
	\label{one}
	\begin{cases}
		S^{\pm} \subseteq \R^N \setminus \{0\} \text{ are Borel sets. }\\
		\cup_{\pm} S^{\pm} = \R^N \setminus \{0\}, \text{ and } \cap_{\pm} S^{\pm} = \emptyset.\\
		\R^N \setminus \{0\} \setminus \mathrm{supp}(\mu_{\lambda} - \mu_{\beta})^{\mp} \subseteq S^{\pm},
	\end{cases}
\end{equation}
and we denote $\mu_{\beta_{\pm}}$  and $\mu_{\lambda_{\pm}}$ as the restrictions of $\mu_{\beta}$ and $\mu_{\lambda}$ to $K^{\pm}$, respectively. Then it is clear to see that
\begin{equation}
	\label{two}
	\begin{cases}
		\mu_{\lambda} = \sum_{\pm}\mu_{\lambda_{\pm}},  \text{ and } \mu_{\beta} = \sum_{\pm}\mu_{\beta_{\pm}} \\
		\pm (\mu_{\lambda_{\pm}} -\mu_{\beta_{\pm}}) = (\mu_{\lambda} - \mu_{\beta})^{\pm}.\\
		\mu_{\lambda_{\pm}}, \mu_{\beta_{\pm}}, \text{ and } \pm (\mu_{\lambda_{\pm}} - \mu_{\beta_{\pm}}) \text{ all are nonnegative Radon measures satisfying \eqref{important}. }
	\end{cases}
\end{equation}
	We will now estimate $J$. Here we follow similar idea as used in estimate of term $J$ in proof of Theorem $\ref{3.6}$, we have
	\begin{align*}
		\mathbb{E}[J]&=\mathbb{E}\int_0^t\int_{(\mathbb{T}^N)^2}\int_{\mathbb{R}^N}\Phi_{\delta}(u(x,s)-v(y,s))[(\tau_z u(x,s)\\
		&\qquad\qquad\qquad-u(x,s))\mu_{\lambda}(z)-(\tau_z v(y,s)-v(y,s))\mu_\beta(z)]\theta_{\epsilon}(x-y)dz dx dy ds\notag\\
		&=:J_r+J^r	
	\end{align*}
	where 
	\begin{align*}
		J_r&=-\mathbb{E}\int_0^t\int_{(\mathbb{T}^N)^2}[g_{x,r}^{\lambda}(u)-g_{y,r}^\beta(v)]\Phi_{\delta}(u(x,s)-v(y,s))\theta_\epsilon(x-y)dx dyds,
	\end{align*}
	and
	\begin{align*}
		J^r&=-\mathbb{E}\int_0^t\int_{(\mathbb{T}^N)^2}[g_{x}^{\lambda,r}(u)-g_{y}^{\beta,r}(v)]\Phi_{\delta}(u(x,s)-v(y,s))\theta_\epsilon(x-y)dx dyds.
	\end{align*}
First we try to get bound on term $J^r$ as follows:
	\begin{align*}
		J^r&=\sum_{\pm}\mathbb{E}\bigg[\int_0^t\int_{(\mathbb{T}^N)^2}\big(g^{\beta_{\pm},r}_{y}(v)-g^{\lambda_{\pm},r}_{x}(u)\big)\Phi_{\delta}(u(x,s)-v(y,s))\theta_{\epsilon}(x-y)dxdy ds\bigg]\\
		&=\mathbb{E}\bigg[\int_0^t\int_{(\mathbb{T}^N)^2}\big(g^{\beta_{+},r}_{x}(u)-g^{\lambda_{+},r}_{x}(u)\big)\Phi_{\delta}(u(x,s)-v(y,s))\theta_{\epsilon}(x-y)dxdyds\bigg]\\
		&\qquad+\mathbb{E}\bigg[\int_0^t\int_{(\mathbb{T}^N)^2}\big(g^{\beta_{+},r}_{y}(v)-g^{\beta_{+},r}_{x}(u)\big)\Phi_{\delta}(u(x,s)-v(y,s))\theta_{\epsilon}(x-y)dxdyds\bigg]\\
		&\qquad+\mathbb{E}\bigg[\int_0^t\int_{(\mathbb{T}^N)^2}\big(g^{\beta_{-},r}_{y}(v)-g^{\lambda_{-},r}_{y}(v)\big)\Phi_{\delta}(u(x,s)-v(y,s))\theta_{\epsilon}(x-y)dxdyds\bigg]\\
		&\qquad+\mathbb{E}\bigg[\int_0^t\int_{(\mathbb{T}^N)^2}\big(g^{\lambda_{-},r}_{y}(v)-g^{\lambda_{-},r}_{x}(u)\big)\Phi_{\delta}(u(x,s)-v(y,s))\theta_{\epsilon}(x-y)dxdyds\bigg]\\
		&:=J_1+J_2+J_3+J_4.
	\end{align*}
	From proof of Theorem \ref{th3.6}, it is clear that $J_2$ and $J_4$ is non-positive. Finally, we are left with two term $J_1$ and $J_3$.
	Consider in the sequel $r_1\,\textgreater\,r$. We have
	\begin{align*}
		J_1&=\mathbb{E}\bigg[\int_0^t\int_{(\mathbb{T}^N)^2}\big(g^{\beta_{+},r,r_1}_{x}(u)-g^{\lambda_{+},r,r_1}_{x}(u)\big)\Phi_{\delta}(u(x,s)-v(y,s))\theta_{\epsilon}(x-y)dxdyds\bigg]\\
		&\qquad+\mathbb{E}\bigg[\int_0^t\int_{(\mathbb{T}^N)^2}\big(g^{\beta_{+},r_1}_{x}(u)-g^{\lambda_{+},r_1}_{x}(u)\big)\Phi_{\delta}(u(x,s)-v(y,s))\theta_{\epsilon}(x-y)dxdyds\bigg]\\
		&:=I_1+I_2.
	\end{align*}
	where the notation $g^{\lambda_{+},r,r_1}$ means that the nonlocal integration is understood in the set $\{r\,\textless\,|z|\,\le r_1\}$ (resp. with $\mu_{\lambda_{+}}$). To estimate the first term of the above equality, we observe that, by construction of the measures, the nonlocal domain of integration is always radial symmetric. Let $\gamma'=\Phi_\delta$, then
	\begin{align*}
		&\int_{(\mathbb{T}^N)^2} g_x^{\beta_{+},r,r_1}(u) \Phi_{\delta}(u(x,s)-v(y,s))\theta_{\epsilon}(x-y)dx dy ds\\
		&=\int_{(\mathbb{T}^N)^2}\int_{r\,\textless\,|z|\,\le r_1}(u(x+z,s)-u(x,s))\Phi_{\delta}(u(x,s)-v(y,s))\theta_{\epsilon}(x-y)\mu_{\beta_{+}}(z)dz dx dy ds\\
		&=\int\limits_{(\mathbb{T}^N)^2}\int\limits_{r\,\textless\,|z|\,\le r_1}[(u(x+z,s)-v(y,s))-(u(x,s)-v(y,s))]\Phi_{\delta}(u(x,s)-v(y,s))\theta_{\epsilon}(x-y)d\mu_{\beta_{+}}(z) dx dy ds\\
	&\le\int_0^t\int_{(\mathbb{T}^N)^2}\int_{r\,\textless\,|z|\,\le r_1}\bigg(\gamma\big(u(x+z,s)-v(y,s)\big)-\gamma\big(u(x,s)-v(y,s)\big)\bigg)\theta_{\epsilon}(x-y)d\mu_{\beta_{+}}(z) dx dy ds\\
		&=-\int_{(\mathbb{T}^N)^2}\gamma\big(u(x,s)-v(y,s)\big)g_x^{\beta_{+},r,r_1}(\theta_{\epsilon})(x-y) dx dy ds
	\end{align*}
	where to derive the pernultimate inequality, we have used the fact that $\gamma(b)-\gamma(a)\,\ge\,\gamma'(a)(b-a)$ with $a=\big(u(x,s)-v(y,s)\big)\,\text{and}\, b= u(x+z)-v(y,s)$. For the last equality, we have performed a change of coordinates for first integral $x\,\to\,x+z, z\to -z$.
	By similar calculation as above, we have
	\begin{align*}
		&I_1\le-\mathbb{E}\bigg[\int_0^t\int_{(\mathbb{T}^N)^2}\gamma\big(u(x,s)-v(y,s)\big)\big(g_x^{\beta_{+},r,r_1}-g^{\lambda_{+},r,r_1}_x\big)\big(\theta_{\epsilon}\big)(x-y) dx dy ds\bigg]\\
		&\le-\mathbb{E}\bigg[\int\limits_0^1\int\limits_0^t\int\limits_{(\mathbb{T}^N)^2}\int\limits_{r\,\textless\,|z|\,\le r_1}(1-\tau)\gamma(u(x,s)-v(y,s))z^T\cdot\text{Hess}_x\theta_\epsilon(x+\tau z -y)\cdot z d(\mu_{\lambda_{+}}-\mu_{\beta_{+}})(z)dxdydsd\tau\bigg]\\
		&\le\,\mathbb{E}\bigg[\int_0^1\int_0^t\int_{(\mathbb{T}^N)^2}\int_{r\,\textless\,|z|\,\le r_1} |\nabla_x \theta_\epsilon(x-\tau z-y)\cdot z d|\nabla_x[\gamma(u(x,s)-v(y,s))].z|(x)d(\mu_{\lambda_{+}}-\mu_{\beta_{+}})(z)dydsd\tau\bigg]\\
		&	\le\,\mathbb{E}\bigg[\int_0^1\int_0^t\int_{(\mathbb{T}^N)^2}\int_{r\,\textless\,|z|\,\le r_1}|z|^2\big|\nabla_x\theta_\epsilon(x-\tau z-y)\big|d\big(|Du(\cdot,t)|\big)(x)d(\mu_{\lambda_{+}}-\mu_{\beta_{+}})(z)dydsd\tau\bigg]\\
		&\le\,\frac{C\,t}{\epsilon}\,\mathbb{E}[\|u_0\|_{BV}] \int_{r\,\textless\,|z|\,\le r_1}|z|^2 d(\mu_{\lambda_{+}}-\mu_{\beta_{+}})(z)\\
		&\le\,\frac{C\,t}{\epsilon}\,\mathbb{E}[\|u_0\|_{BV}] \int_{|z|\,\le r_1}|z|^2 d(\mu_{\lambda_{+}}-\mu_{\beta_{+}})(z),
	\end{align*}
	where we have used the fact that for any Lipschitz continous function $\gamma$ (with Lipschitz constant 1), $|D\gamma(u)|\,\le\,|Du|$. On the other hand, to handle the other term we proceed as follows:
	\begin{align*}
		I_2 &\le\,C\,t\,\int_{|z|\textgreater\,r_1}\mathbb{E}\|u_0(\cdot+z)-u_0\|_{L^1(\mathbb{T}^N)}\,d(\mu_{\lambda_{+}}-\mu_{\beta_{+})},
	\end{align*}
	thanks to the contraction principle of Theorem \ref{th3.6} and the fact that $u(\cdot,\cdot+z)$ is the solution correspoing to the initial condition $u_0(\cdot+z)$.
	Note that exact same calculations will help us to estimate $J_3$. Indeed, we have
	\begin{align*}
		J_3&\le\,\frac{C\,t}{\epsilon}\mathbb{E}[\|v_0\|_{BV}]\int_{\int_{|z|\,\le\,r_1}}|z|^2 d(\mu_{\beta_{-}}-\mu_{\lambda_{-}})(z)+C\,t\,\int_{|z|\textgreater\,r_1}\mathbb{E}\|v_0(\cdot+z)-v_0\|_{L^1(\mathbb{T}^N)}\,d(\mu_{\beta_{-}}-\mu_{\lambda_{-}})(z),
	\end{align*} 
	Now, we will try to estimate $J_r$ when $r\to0.$
	First note that,
	\[
	|g_{x,r}^\lambda(\theta_{\epsilon})|\le 
	\begin{cases}
		\|D\theta_{\epsilon}\|_{L^\infty} \int_{|z|\le r} \frac{|z|}{|z|^{d+2\lambda}}\,dz, & \lambda \in (0,1/2) \\[2mm]
		\|D^2\theta_{\epsilon}\|_{L^\infty}\int_{|z|\le r} \frac{|z|^2}{|z|^{d+2\lambda}}\,dz, & \lambda \in [1/2,1)
	\end{cases}
	\]
	Thus we see that in both cases $|g_{x,r}^\lambda(\theta_\epsilon)| \le c r^s$ for some $s>0$ and  $\lim_{r\to 0}|g_{x,r}^\lambda(\theta_\epsilon)|=0$.\\
	Making of use of caclution as we used in estimating term $I_1$, we have
	\begin{align*}
		|J_r|&\,\le\,\mathbb{E}\bigg[\int_0^t\int_{(\mathbb{T}^N)^2}|\gamma\big(u(x,s)-v(y,s)\big)|\,|g_{x,r}^\lambda(\theta_\epsilon)(x-y)|+|g_{y,r}^\beta(\theta_{\epsilon})(x-y)|dx dy ds\bigg]\\
		&\le\mathbb{E}\bigg[\int_0^t\int_{(\mathbb{T})^2}|u(x,s)-v(y,s)|\big(|g_{x,r}^\lambda(\theta_{\epsilon})(x-y)|+|g_{y,r}^\beta(\theta_{\epsilon})(x-y)|\big)dx dy ds\bigg]\underset{ r \to 0} \longrightarrow 0.\\
	\end{align*} 
\textbf{Step 3:} We estimate $\mathcal{R}_{\kappa}$ as follows: $\mathbb{P}$-almost surely, for all $t\in[0,T]$,
\begin{align*}
	\mathcal{R}_{\kappa}\notag&=\frac{1}{2}\int_{(\mathbb{T}^N)^2}\theta_{\epsilon}(x-y)\int_0^t\int_{\mathbb{R}^2}\kappa_{\delta}(\xi-\zeta)|\Phi(\xi)-\Psi(\xi)+\Psi(\xi)-\Psi(\zeta)|^2 d\mathcal{V}_{x,s}^{1}\oplus\mathcal{V}_{y,s}^{2}(\xi,\zeta)dx dy ds\\
	&\le\,C\,t\,(\delta^{-1}\|\Phi-\Psi\|_{L^{\infty}}^2+\delta^{\lambda_{G_2}}). 
\end{align*}
	\textbf{Step 4:} In this step we estimate K term by using similar argument as estimate $K$ term in proof of Theorem \ref{comparison}. Since
	\begin{align*}A(\xi)+B(\zeta)=\big(\sigma(\xi)-\tau(\zeta)\big)\big(\sigma(\xi)-\tau(\zeta)\big)+\big(\sigma(\xi)\tau(\zeta)+\tau(\zeta)\sigma(\xi)\big).
	\end{align*}
	We have $\mathbb{P}$-almost surely, for all $t\in[0,T]$
	\begin{align*}
		&\int_0^t\int_{(\mathbb{T}^N)^2}\int_{\mathbb{R}^2}f_1(x,s,\xi) \bar{f}_2(y,s,\zeta)\big(\sigma(\xi)\tau(\zeta)+\tau(\zeta)\sigma(\xi)\big):D_{x}^2\theta_\epsilon(x-y)\kappa_\delta(\xi-\zeta) d\xi d\zeta dx dy ds\\
		&=2\int_0^t\int_{(\mathbb{T}^N)^2}\theta_{\epsilon}(x-y)\phi_{\delta}(u-v)\times \mbox{div}_{x}\int_0^{u}\sigma(\xi)d\xi\cdot\mbox{div}_{y}\int_0^{v}\tau(\zeta)d\zeta dx dy ds,
	\end{align*}
	and by using chain rule \ref{chain} we get
	\begin{align*}
		&-\int_0^t\int_{(\mathbb{T}^N)^2}\int_{\mathbb{R}^2}\theta_\epsilon(x-y)\kappa_\delta(\xi-\zeta) d\mathcal{V}_{x,s}^1(\xi)dx d\eta_{v,3}(y,s,\zeta)\\&\qquad-\int_0^t\int_{(\mathbb{T}^N)^2}\int_{\mathbb{R}^2}\theta(x-y)\kappa(\xi-\zeta)d\mathcal{V}_{y,s}^2(\zeta)\,dy\,d\eta_{u,3}(x,s,\xi)\\
		&=-\int_0^t\int_{(\mathbb{T}^N)^2}\theta_{\epsilon}(x-y)\phi_{\delta}(u_1-u_2)|\mbox{div}_y\int_0^{u_2}\sigma(\zeta)d\zeta|^2dx dy ds\notag\\
		&\qquad-\int_0^t\int_{(\mathbb{T}^N)^2}\theta_{\epsilon}(x-y)\phi_{\delta}(u_1-u_2)|\mbox{div}_x\int_0^{u_1}\sigma(\xi)d\xi|^2 dx dy ds.
	\end{align*}
	So we can conclude that
	\begin{align*}
		K&\le\,\int_0^t\int_{(\mathbb{T}^N)^2}\int_{\mathbb{R}^2}f_1(x,s,\xi) \bar{f}_2(y,s,\zeta)\big(\sigma(\xi)-\tau(\zeta)\big)\big(\sigma(\xi)-\tau(\zeta)\big):D_{x}^2\theta_\epsilon(x-y)\kappa_\delta(\xi-\zeta) d\xi d\zeta dx dy ds\\
		&=\int_0^t\int_{(\mathbb{T}^N)^2}\int_{\mathbb{R}^2}f_1(x,s,\xi) \bar{f}_2(y,s,\zeta)\big(\sigma(\xi)-\tau(\xi)\big)\big(\sigma(\xi)-\tau(\xi)\big):D_{x}^2\theta_\epsilon(x-y)\kappa_\delta(\xi-\zeta) d\xi d\zeta dx dy ds\\
		&\qquad+\int_0^t\int_{(\mathbb{T}^N)^2}\int_{\mathbb{R}^2}f_1(x,s,\xi) \bar{f}_2(y,s,\zeta)\big(\sigma(\xi)-\tau(\xi)\big)\big(\tau(\xi)-\tau(\zeta)\big):D_{x}^2\theta_\epsilon(x-y)\kappa_\delta(\xi-\zeta) d\xi d\zeta dx dy ds\\
		&\qquad+\int_0^t\int_{(\mathbb{T}^N)^2}\int_{\mathbb{R}^2}f_1(x,s,\xi) \bar{f}_2(y,s,\zeta)\big(\tau(\xi)-\tau(\zeta)\big)\big(\sigma(\xi)-\tau(\xi)\big):D_{x}^2\theta_\epsilon(x-y)\kappa_\delta(\xi-\zeta) d\xi d\zeta dx dy ds\\
		&\qquad+\int_0^t\int_{(\mathbb{T}^N)^2}\int_{\mathbb{R}^2}f_1(x,s,\xi) \bar{f}_2(y,s,\zeta)\big(\tau(\xi)-\tau(\zeta)\big)\big(\tau(\xi)-\tau(\zeta)\big):D_{x}^2\theta_\epsilon(x-y)\kappa_\delta(\xi-\zeta) d\xi d\zeta dx dy ds\\
		&=:K_1+K_2+K_3+K_4
	\end{align*}
	By similar calculation as proof in Theorem \ref{comparison}, we can get the following estimate
	\begin{align*}K_1\,&\le\,C\,t\,\epsilon^{-2}\|\sigma-\tau\|_{L^\infty(\mathbb{R})}^2,\\
K_2\,&\le\,C\,t\|\sigma-\tau\|_{L^\infty(\mathbb{R})}\,\delta^{\gamma_b}\,\epsilon^{-2},\\
K_3\,&\le\,C\,t\,\|\sigma-\tau\|_{L^\infty(\mathbb{R})}\,\delta^{\gamma_b}\,\epsilon^{-2},\\
K_4\,&\le\,C\,t\,\delta^{2\gamma_b}\,\epsilon^{-2}.
	\end{align*}
	\textbf{Step 5:}
	From previous steps we have for all $t\in[0,T]$
	\begin{align}\label{final}
		&\mathbb{E}\int_{(\mathbb{T}^N)^2}\int_{(\mathbb{R})^2}\theta_{\epsilon}(x-y)\kappa_{\delta}(\xi-\zeta)f_1(x,t,\xi)\bar {f}_2(y,t,\zeta)d\xi d\zeta dx dy\notag\\
		&\le\mathbb{E}\int_{(\mathbb{T}^N)}\int_{(\mathbb{R}^2)}\theta_{\epsilon}(x-y)\kappa_{\delta}(\xi-\zeta) f_{1,0} (x,\xi)\bar f_{2,0} (y,\zeta)d\xi d\zeta dx dy+C\,t\,\|F'-G\|_{L^\infty(R)}\mathbb{E}[\|v_0\|_{BV}]\notag\\
		&\qquad+C\,t\, \epsilon^{-1}\,\delta^{\lambda_{G_1}}+C\,t\,(\delta^{-1}\|\Phi-\Psi\|_{L^{\infty}}^2+\delta^{\lambda_{G_2}})+\frac{C\,t}{\epsilon}\mathbb{E}[\|u_0\|_{BV}+\|v_0\|_{BV}]\int_{|z|\,\le\,r_1}|z|^2 d|\mu_{\lambda}-\mu_{\beta}|(z)\notag\\
		&\,\qquad+C\,t\,\int_{|z|\,\textgreater\,r_1}\mathbb{E}\big(\|u_0(\cdot+z)-u_0\|_{L^1(\mathbb{T}^N}+\|v_0(\cdot+z)-v_0\|_{L^1(\mathbb{T}^N)}\big)d|\mu_\lambda-\mu_{\beta}|(z)\notag\\&\qquad+C\,t\,\epsilon^{-2}\|\sigma-\tau\|_{L^\infty(\mathbb{R})}^2+C\,t\,\|\sigma-\tau\|_{L^\infty(\mathbb{R})}\,\delta^{\gamma_b}\,\epsilon^{-2}+C\,t\,\delta^{2\gamma_b}\,\epsilon^{-2}.
	\end{align}
	Finally, we will now estimate $\mathbb{E}\int_{\mathbb{T}^N}(u(x)-v(x))_{+} dx$\,\, in terms of given data as follows:
	\begin{align*}
		&\mathbb{E}\int_{\mathbb{T}^N}(u(x,t)-v(x,t))_{+} dx=\mathbb{E}\int_{\mathbb{T}^N}\int_{\mathbb{R}^2}f_1(x,t,\xi)\bar{ f}_2(x,t,\zeta)d\xi dx\\
		=&\mathbb{E}\int_{(\mathbb{T}^N)^2}\int_{\mathbb{R}^2}\theta_{\epsilon}(x-y)\kappa_\delta(\xi-\zeta)f_1(x,t,\xi)\bar{f}_2(y,t,\zeta) d\xi d\zeta dx dy+\eta_t(u,v,\epsilon,\delta).
	\end{align*}
	where
	\begin{align*}
		\eta_t(u,v,\epsilon,\delta)&=\mathbb{E}\int_{\mathbb{T}^N}\int_{\mathbb{R}^2}f_1(x,t,\xi)\bar{ f}_2(x,t,\zeta)d\xi dx\\
		&\qquad\qquad\qquad-\mathbb{E}\int_{(\mathbb{T}^N)^2}\int_{\mathbb{R}^2}\theta_{\epsilon}(x-y)\kappa_\delta(\xi-\zeta)f_1(x,t,\xi)\bar{f}_2(y,t,\zeta) d\xi d\zeta dx dy\\
		&=\bigg(\mathbb{E}\int_{\mathbb{T}^N}\int_{\mathbb{R}^2}f_1(x,t,\xi)\bar{ f}_2(x,t,\zeta)d\xi dx\\
		&\qquad\qquad\qquad\qquad-\mathbb{E}\int_{(\mathbb{T}^N)^2}\int_{\mathbb{R}}\theta_\epsilon(x-y)f_1(x,t,\xi)\bar{ f}_2(y,s,\xi)d\xi dx dy\bigg)\\
		&\qquad\qquad+\bigg(\mathbb{E}\int_{(\mathbb{T}^N)^2}\int_{\mathbb{R}}\theta_\epsilon(x-y)f_1(x,t,\xi)\bar{ f}_2(y,s,\xi)d\xi dxdy\\
		&\qquad\qquad\qquad-\mathbb{E}\int_{(\mathbb{T}^N)^2}\int_{\mathbb{R}^2}\theta_{\epsilon}(x-y)\kappa_\delta(\xi-\zeta)f_1(x,t,\xi)\bar{f}_2(y,t,\zeta) d\xi d\zeta dx dy\bigg)\\
		&=:H_1+H_2
	\end{align*}
where
	\begin{align}\label{H1}
		|H_1|&=\bigg|\mathbb{E}\int_{(\mathbb{T}^N)^2}\theta_\epsilon(x-y)\int_{\mathbb{R}}\mathbbm{1}_{u(x)\textgreater\xi}\big[\mathbbm{1}_{u(x)\le\,\xi}-\mathbbm{1}_{v(y)\le\,\xi}\big]d\xi dx dy\bigg|\notag\\
		&=\bigg|\mathbb{E}\int_{(\mathbb{T}^N)^2}\theta_{\epsilon}(x-y)(v(x)-v(y))dx dy\bigg|\notag\\&\le\,\epsilon\,\mathbb{E}[\|v_0\|_{BV}],
	\end{align}
	The error term will estimate as follows:
	\begin{align*}
		&\bigg|\int_{(\mathbb{T}^N)^2}\int_{\mathbb{R}}\theta_{\epsilon}(x-y)f_1(x,t,\xi)\bar{f}_2(y,t,\xi) d\xi dx dy\\
		& -\int_{(\mathbb{T}^N)^2}\int_{\mathbb{R}}\theta_{\epsilon}(x-y)\kappa_{\delta}(\xi-\zeta)f_1(x,t,\xi)\bar{f}_2(y,t,\zeta)d\xi d\zeta dx dy\bigg|\\
		&=\bigg|\int_{(\mathbb{T}^N)^2}\theta_{\epsilon}(x-y)\int_{\mathbb{R}}\mathbbm{1}_{u(x)\textgreater\xi}\int_{\mathbb{R}}\kappa_{\delta}(\xi-\zeta)\big[\mathbbm{1}_{v(y)\le\xi}-\mathbbm{1}_{v(y)\le\zeta}\big]d\zeta d\xi dx dy\bigg|\\
		&\le\int_{(\mathbb{T}^N)^2}\int_{\mathbb{R}}\theta_{\epsilon}(x-y)\mathbbm{1}_{u(x)\textgreater\xi}\int_{\xi-\delta}^\xi\kappa_\delta(\xi-\zeta)\mathbbm{1}_{\zeta\textless v(y)\le\xi} d\zeta d\xi dx dy\\
		&\qquad+\int_{(\mathbb{T}^N)^2}\int_{\mathbb{R}}\theta_\epsilon(x-y)\mathbbm{1}_{u(x)\textgreater\xi}\int_{\xi}^{\xi+\delta}\kappa_{\delta}(\xi-\zeta)\mathbbm{1}_{\zeta\textless v(y)\le\zeta} d\zeta d\xi dx dy\\
		&\le \frac{1}{2}\int_{(\mathbb{T}^N)^2}\theta_{\epsilon}(x-y)\int_{v(y)}^{min\{u(x),v(y)+\delta\}} d\xi dx dy+\frac{1}{2}\int_{(\mathbb{T})^2}\theta_\epsilon(x-y)\int_{v(y)-\delta}^{min\{u(x),v(y)\}}d\xi dx dy\\
		&\le \delta.
	\end{align*}
	It gives that 
	\begin{align}\label{H2}
		|H_2|\le\,\delta,
	\end{align}
and
	\begin{align}\label{E.5}
		&\eta_t(u,v,\epsilon, \delta)\le\,\epsilon\,\mathbb{E}[\|v_0\|_{BV}]+\delta
	\end{align}
	Finally, we conclude that for all $t\in[0,T]$
	\begin{align*}
		\mathbb{E}\int_{\mathbb{T}^N}&(u(x,t)-v(x,t))_{+}dx
		\le\,\mathbb{E}\bigg[\int_{(\mathbb{T}^N)^2}\int_{\mathbb{R}^2} \theta_\epsilon(x-y)\kappa_{\delta}(\xi-\zeta) f_{1,0}(x,\xi)\bar{ f}_{2,0}(y,\zeta) d\xi d\zeta dx dy \bigg]\\
		&\qquad\qquad\qquad+\epsilon\,\mathbb{E}[\|v_0\|_{BV}]+\delta+C\,t\,\|F'-G'\|_{L^\infty(R)}\mathbb{E}[\|v_0\|_{BV}]+C\,t\, \epsilon^{-1}\,\delta^{\lambda_{G_1}}\\&\qquad\qquad\qquad+C\,t\,(\delta^{-1}\|\Phi-\Psi\|_{L^{\infty}}^2+\delta^{\lambda_{G_2}})+\frac{C\,t}{\epsilon}\big(\mathbb{E}[\|u_0\|_{BV}+\|v_0\|_{BV}]\big)\int_{|z|\,\le\,r_1}|z|^2 d|\mu_{\lambda}-\mu_{\beta}|(z)\\
		&\,\qquad\qquad\qquad+C\,t\int_{|z|\,\textgreater\,r_1}\mathbb{E}\big(\|u_0(\cdot+z)-u_0\|_{L^1(\mathbb{T}^N)}+\|v_0(\cdot+z)-v_0\|_{L^1(\mathbb{T}^N)}\big)d|\mu_\lambda-\mu_{\beta}|(z),\\
		&\qquad\qquad\qquad+C\,t\,\epsilon^{-2}\|\sigma-\tau\|_{L^\infty(\mathbb{R})}^2+C\,t\,\|\sigma-\tau\|_{L^\infty(\mathbb{R})}\,\delta^{\gamma_b}\,\epsilon^{-2}+C\,t\,\delta^{2\gamma_b}\,\epsilon^{-2}.
	\end{align*}
It shows that for all $t\in[0,T]$
	\begin{align*}
		\mathbb{E}\int_{\mathbb{T}^N}|u(x,t)-v(x,t)|dx
		&\le\,C\,\bigg(\mathbb{E}\bigg[\int_{\mathbb{T}^N}|v_0(x)-u_0(x)|\,dx\bigg] + 2\epsilon\,\big(\mathbb{E}[\|v_0\|_{BV}+\|u_0\|_{BV}]\big)+2\delta\\&\qquad+t\,\|F'-G'\|_{L^\infty(R)}\mathbb{E}[\|v_0\|_{BV}]+ t\, \epsilon^{-1}\,\delta^{\lambda_{G_1}} + t\,(\delta^{-1}\|\Phi-\Psi\|_{L^{\infty}}^2+\delta^{\lambda_{G_2}})\\&\qquad+\frac{\mathbb{E}[\|u_0\|_{BV}+\|v_0\|_{BV}]\,t}{\epsilon}\int_{|z|\,\le\,r_1}|z|^2 d|\mu_{\lambda}-\mu_{\beta}|(z)\\
		&\qquad+t\,\int_{|z|\,\textgreater\,r_1}\mathbb{E}\big(\|u_0(\cdot+z)-u_0\|_{L^1(\mathbb{T}^N)}+\|v_0(\cdot+z)-v_0\|_{L^1(\mathbb{T}^N)}\big)d|\mu_\lambda-\mu_{\beta}|(z)\\&\qquad + t\epsilon^{-2}\|\sigma-\tau\|_{L^\infty(\mathbb{R})}^2+ t\|\sigma-\tau\|_{L^\infty(\mathbb{R})}\,\delta^{\gamma_b}\,\epsilon^{-2}+t\,\delta^{2\gamma_b}\,\epsilon^{-2}\bigg).	
	\end{align*}
We can choose $\lambda_2=\max\{\frac{2}{\lambda_{G_1}},\frac{2}{\gamma_b},2\}$ and set
	$$\delta=\epsilon^{\lambda_2}=\|\Phi-\Psi\|_{L^\infty(\mathbb{\mathbb{R}})}+\sqrt{\int_{|z|\,\le\,r_1}|z|^2 d|\mu_{\lambda}-\mu_{\beta}|(z)}+\|\sigma-\tau\|_{L^\infty(\mathbb{R})}.$$
	If we assume differences are small, then we can conclude that for all $t\in[0,T]$
	\begin{align*}
		\mathbb{E}&\int_{\mathbb{T}^N}|u(x,t)-v(x,t)|dx \le\,C_T\,\bigg(\mathbb{E}\bigg[\int_{\mathbb{T}^N}|v_0(x)-u_0(x)|dx\bigg]+\|F'-G'\|_{L^\infty(\mathbb{R})}\\&\qquad+\bigg(\|\Phi-\Psi\|_{L^\infty(\mathbb{R})}+\sqrt{\int_{|z|\,\le\,r_1}|z|^2 d|\mu_{\lambda}-\mu_{\beta}|(z)}+\|\sigma-\tau\|_{L^\infty(\mathbb{R})}\bigg)^{\min\big\{\frac{1}{2},\frac{\lambda_{G_1}}{2},\lambda_{G_2}, \frac{\gamma_b}{2}\big\}}\\&\qquad+\int_{|z|\,\textgreater\,r_1}\mathbb{E}\big(\|u_0(\cdot+z)-u_0\|_{L^1(\mathbb{T}^N)}+\|v_0(\cdot+z)-v_0\|_{L^1(\mathbb{T}^N)}\big)d|\mu_\lambda-\mu_{\beta}|(z).
	\end{align*}
	This finishes the proof of Theorem \ref{main theorem 2}.
	
	\appendix
	
	\section{Derivation of the kinetic formulation.}\label{A}
	In this Appendix, we briefly derive the kinetic formulation of equation \eqref{1.1} with Lipschitz flux, because we work with  the approximations of equation \eqref{1.1}.  We show that if u is a weak solution to \eqref{1.1} such that 
	$ u\in L^2(\Omega ; C([0,T];L^2(\mathbb{T}^N)))\cap L^2(\Omega;L^2(0,T; H^1(\mathbb{T}^N))$ then $f(t)=\mathbbm{1}_{u(t)\textgreater\xi}$ satisfies
	$$df(t)+ F'\cdot\nabla f(t) dt -A:D^2 f(t) dt+ g_x^\lambda(f(t)) dt =\delta_{u(t)=\xi} \Phi dW(t) + \partial_{\xi} \big(\eta_1-\frac{1}{2} \beta^2 \delta_{u(t)=\xi}\big) dt$$
	in the sense of $\mathcal{D}'(\mathbb{T}^N\times\mathbb{R})$ where
	$$d\eta_1(x,t,\xi)=\int_{\mathbb{R}^N}|u(x+z)-\xi|\mathbbm{1}_{Conv\{u(x,t),u(x+z,t)\}}(\xi) \mu(z) dz d\xi dx dt .$$
	Indeed, it follows from generalized It\^o formula \cite[Appendix A]{vovelle}, for $\phi\in C_b^2(\mathbb{R})$ with $\phi(-\infty)=0$, $\kappa\in C^2(\mathbb{T}^N),$ $\mathbb{P}$-almost surely,
	\begin{align*}
		\langle \phi(u(t)), \kappa \rangle &= \langle \phi(u_0), \kappa \rangle -\int_0^t \langle \phi'(u) \mbox{div} (F(u)), \kappa \rangle ds-\int_0^t \langle \phi'(u)g_x^\lambda(u),\kappa \rangle ds\\
		&\qquad-\int_0^t\langle \phi''(u)\nabla u \cdot (A(u)\nabla u), \kappa \rangle ds+ \int_0^t \langle  \mbox{div}(\phi'(u)A(u)\nabla u),\kappa \rangle ds\\
		&\qquad+\sum_{k \ge 1} \int_0^t \langle \phi'(u) \beta_k(x,u),\kappa \rangle dw_k(s)+\frac{1}{2}\int_0^t\langle \phi''(u)\beta^2(x,u),\kappa \rangle ds.
	\end{align*}
	Note that, for $u\in H^{\lambda}(\mathbb{T}^N)$
	$$g_x^\lambda(u)(\cdot)=-\lim_{r\to0}\int_{|z|\textgreater r} (u(\cdot+z)-u(\cdot)) \mu(z) dz$$ with limit in $L^2(\mathbb{T}^N).$ By using it,  we have $\p$-almost surely, almost $t\in[0,T]$,
	\begin{align*}\langle \phi'(u(x,t)) g_x^\lambda&(u(x,t)),\kappa(x) \rangle \\
		&= -\lim_{r\to0}\langle \phi'(u(x,t)) \int_{|z|\textgreater r}(u(x+z,t)-u(x,t)) \mu(z) dz ,\kappa(x) \rangle, 
	\end{align*}
By making use of Taylor's identity, we have
	\begin{align*}
		&\langle \phi'(u(x,t))\int_{|z|\textgreater r}(u(x+z,t)-u(x,t))\mu(z) dz ,\kappa \rangle \\ 
		&\qquad\qquad= \int_{\mathbb{T}^N}\int_{|z|\textgreater r} \phi'(u(x,t))(u(x+z,t)-u(x,t))\kappa(x) \mu(z) dz dx\\
		&\qquad\qquad=\int_{\mathbb{T}^N}\int_{|z|\textgreater r}\kappa(x)\bigg(\int_{\mathbb{R}}(\phi'(\xi)\mathbbm{1}_{u(x+z,t)\textgreater\xi}-\phi'(\xi)\mathbbm{1}_{u(x,t)\textgreater\xi})d\xi\\
		&\qquad\qquad\qquad-\int_{\mathbb{R}}\phi''(\xi)|u(x+z,t)-\xi|\mathbbm{1}_{Conv\{u(x+z,t),u(x,t)\}} \bigg)\mu(z)dz dx\\
		&\qquad\qquad=\int_{\mathbb{T}^N}\int_{|z|\textgreater r}\int_{\mathbb{R}} \phi(\xi) \mathbbm{1}_{u(x,t)\textgreater\xi}(\kappa(x+z)-\kappa(x)))\mu(z)\,d\xi\,dz dx\\
		&\qquad\qquad\qquad-\int_{\mathbb{T}^N} \int_{\mathbb{R}}\kappa(x) \phi''(\xi)\int_{|z|\textgreater r} |u(x+z,t)-\xi|\mathbbm{1}_{Conv\{u(x+z,t),u(x,t)\}}(\xi)\mu(z)dz d\xi dx,
	\end{align*}
It implies that
	\begin{align*}
		&\langle \phi'(u(x,t)) g_x^\lambda(u)(x,t),\kappa \rangle\\
		&=-\lim_{r\to0}\int_{\mathbb{T}^N} \int_{\mathbb{R}} \phi'(\xi) \mathbbm{1}_{u(x,t)\textgreater\xi}\int_{|z|\textgreater r}(\kappa(x+z)-\kappa(x))\mu(z)d\xi dz dx\\
		&\quad+\lim_{r\to0}\int_{\mathbb{T}^N} \int_{\mathbb{R}}\kappa(x) \phi''(\xi)\int_{|z|\textgreater r} |u(x+z,t)-\xi|\mathbbm{1}_{Conv\{u(x+z,t),u(x,t)\}}(\xi)\mu(z)dz d\xi dx\\
		&=-\int_{\mathbb{T}^N} \int_{\mathbb{R}}\phi'(\xi) \mathbbm{1}_{u(x,t)\textgreater\xi}\lim_{r\to 0}\int_{|z|\textgreater r}(\kappa(x+z)-\kappa(x))\mu(z)d\xi dz dx\\
		&\quad+\int_{\mathbb{T}^N} \int_{\mathbb{R}}\kappa(x) \phi''(\xi)\lim_{r\to 0}\int_{|z|\textgreater r} |u(x+z,t)-\xi|\mathbbm{1}_{Conv\{u(x+z,t),u(x,t)\}}(\xi)\mu(z)dz d\xi dx\\
		&=\int_{\mathbb{T}^N}\int_{\mathbb{R}}\phi'(\xi) \mathbbm{1}_{u(x,t)\textgreater\xi}g_x^\alpha(\kappa)(x)dxd\xi\\
		&\qquad\qquad+\int_{\mathbb{T}^N}\int_{\mathbb{R}}\kappa(x)\phi''(\xi)\int_{\mathbb{R}^N}|u(x+z,t)-\xi|\mathbbm{1}_{Conv\{u(x+z,t),u(x,t)\}}(\xi)\mu(z) dz d\xi dx\\
		&=\langle \mathbbm{1}_{u(x,t)\textgreater \xi } \phi'(\xi), g_x^\lambda(\kappa)(x)\rangle_{x,\xi } - \langle \kappa(x) \phi'(\xi),\partial_{\xi}\eta_1\rangle_{x,\xi}
	\end{align*}
	where taking limit inside intergral in second term in right hand side is justified by the fact that, $\mathbb{P}$-almost surely, $\eta_1\in L^1(\mathbb{T}^N\times\mathbb{R}\times[ 0,T])$. Afterward, we proceed remaining term and apply the chain rule for functions from Sobolev spaces. We obtain the following identity that hold true in $\mathcal{D}'(\mathbb{T}^N),$
	\begin{align*}\langle \mathbbm{1}_{u(x,t)\textgreater\xi} , \phi'\rangle_\xi&= \int_{\mathbb{R}}\mathbbm{1}_{u(x,t)\textgreater\xi}\phi'(\xi)d\xi= \phi(u(x,t)),\\
	\phi'(u(x,t))\mbox{div}(F(u(x,t)))&=\phi'(u(x,t))F'(u(x,t)).\nabla u(x,t)\\
	&=\mbox{div}(\int_{-\infty}^{u(x,t)}F'(\xi)\phi'(\xi))= \mbox{div}(\langle F'\mathbbm{1}_{u(x,t)\textgreater\xi},\phi'\rangle_\xi),\\ 
	\phi''(u)\nabla u\cdot(A(u)\nabla u)&=-\langle\partial_{\xi} \eta_2, \phi' \rangle_\xi,\\
	\mbox{div}(\phi'(u)A(u)\nabla u)&=D^2:\bigg(\int_{-\infty}^u A(\xi) \phi'(\xi) d\xi\bigg)= D^2:\langle A\mathbbm{1}_{u(x,t)\,\textgreater\,\xi}, \phi' \rangle_{\xi},\\
	\phi'(u(x,t))\beta_k(x,u(x,t))&=\langle \beta_k \delta_{u(x,t)=\xi} ,\phi'\rangle_{\xi},\\
	\phi''(u(x,t))\beta^2(x,u(x,t))&=\langle \beta^2 \delta_{u(x,t)=\xi} ,\phi''\rangle_{\xi}=-\langle \partial_{\xi}(\beta^2\delta_{u(x,t)=\xi}), \phi' \rangle_{\xi}.
	\end{align*}
	
	Therefore, we define  $\phi=\int_{-\infty}^\xi \theta(\zeta) d\zeta$, for some $\theta\in C_c^\infty(\mathbb{R})$  to obtain the result.
	
	\subsection*{Acknowledgments}
	The author wishes to thank Ujjwal Koley for many stimulating discussions and valuable suggestions.

\end{document}